\documentclass[12pt, reqno]{amsart}
\usepackage[utf8]{inputenc}
\usepackage{amsmath}
\usepackage{amsfonts}
\usepackage{amssymb}
\usepackage{amsthm}
\usepackage{bm}
\usepackage{bbm}
\usepackage{xcolor}
\usepackage{url}
\usepackage{hyperref}
\usepackage{mathtools}
\usepackage{enumitem}
\usepackage{array}   

\newcolumntype{L}{>{$}l<{$}} 

\numberwithin{equation}{section}

\theoremstyle{plain}
\newtheorem{theorem}{Theorem}[section]
\newtheorem{lemma}[theorem]{Lemma}
\newtheorem{corollary}[theorem]{Corollary}
\newtheorem{proposition}[theorem]{Proposition}
\newtheorem{conjecture}[theorem]{Conjecture}
\newtheorem{remark}[theorem]{Remark}
\newtheorem{definition}[theorem]{Definition}

\newcommand{\Z}{\mathbb{Z}}
\newcommand{\Q}{\mathbb{Q}}
\newcommand{\R}{\mathbb{R}}
\newcommand{\C}{\mathbb{C}}
\newcommand{\N}{\mathbb{N}}

\newcommand{\F}{\mathbb{F}}

\providecommand{\ceil}[1]{\left \lceil #1 \right \rceil }
\providecommand{\floor}[1]{\left \lfloor #1 \right \rfloor }
\providecommand{\abs}[1]{\left \lvert #1 \right \rvert}
\providecommand{\norm}[1]{\left \lVert #1 \right \rVert}

\newcommand{\im}{\operatorname{Im}}
\newcommand{\gal}{\operatorname{Gal}}
\newcommand{\aff}{\operatorname{Aff}}
\newcommand{\sgn}{\operatorname{Sgn}}
\newcommand{\aut}{\operatorname{Aut}}
\newcommand{\proj}{\operatorname{P}}

\title[Lebesgue--Nagell equation]{On the Lebesgue--Nagell equation $x^2-2 = y^p$}
\author{Ethan Katz}
\author{Kyle Pratt}
\address{Brigham Young University, Department of Mathematics, Provo, UT 84602, USA}
\email{ethanhkatz@gmail.com}
\email{kyle.pratt@mathematics.byu.edu}

\subjclass[2010]{}

\allowdisplaybreaks

\begin{document}
\date{}

\begin{abstract}
We investigate the Lebesgue--Nagell equation $$x^2-2=y^p$$
in integers $x,y,p$ with $p\geq 3$ an odd prime. A longstanding folklore conjecture asserts that the only solutions are the ``trivial'' ones with $y=-1$. We confirm the conjecture unconditionally for $p\leq 13$, and prove the conjecture holds for $p>911$ through a careful application of lower bounds for linear forms in two logarithms. We also show that any ``nontrivial'' solution must satisfy $y > 10^{1000}$. In addition, we establish auxiliary results that may support future progress on the problem, and we revisit some prior claims in the literature.
\end{abstract}

\maketitle

\section{Introduction}

Diophantine equations are one of the most ancient and studied topics in number theory. In part, mathematicians are fascinated by Diophantine equations because they are both simple to state and yet often extremely difficult to solve. Famous equations like the Fermat equation $x^n = y^n + z^n, n\geq 3$, or the Catalan equation 
\begin{align}\label{eq:catalan eq}
  x^n + 1 = y^m, \ \ m,n \geq 2
\end{align}
 show the value in studying Diophantine equations, as studying these equations led mathematicians to develop powerful techniques that are now central to modern number theory.

Positive integer solutions $(x,y)$ to the Catalan equation \eqref{eq:catalan eq} give rise to consecutive perfect powers (gaps between perfect powers of size one), and {Mih\u ailescu} \cite{catalansequation2004} famously showed the only solution to \eqref{eq:catalan eq} is $2^3+1=3^2$. While Catalan's equation has now been solved, the more general problem of bounding the gaps between perfect powers is not well understood. For instance, Pillai's conjecture that there are only finitely many solutions to $x^n + d = y^m$ for any fixed $d$ remains wide open \cite[p. 253--254]{rationalnumbertheory}. 

A weaker, though still very challenging, version of Pillai's conjecture is just to consider the gaps between perfect squares and higher powers. Thus, one is led to consider solutions to the equation
\begin{align}\label{eq:general lebesgue nagell}
x^2 + d = y^n, \ \ n\geq 2,
\end{align}
in integers $x$ and $y$, where $d$ is a fixed, nonzero integer. A Diophantine equation of the form \eqref{eq:general lebesgue nagell} is known as a \emph{Lebesgue--Nagell equation}. There is a vast literature relating to Lebesgue--Nagell equations (see the recent survey \cite{LS2020} for a partial introduction to the subject). It is known that, for fixed $d$, the equation \eqref{eq:general lebesgue nagell} has at most finitely many solutions \cite{ShoreyTijdeman1986}. Moreover, one can use techniques like linear forms in logarithms to obtain effective upper bounds on the size of $x,y,n$. Unfortunately, these bounds are far too large for a direct computation to resolve \eqref{eq:general lebesgue nagell}, so one needs additional techniques to solve these equations completely.

Equation \eqref{eq:general lebesgue nagell} has been solved for all $1 \leq d \leq 100$ \cite{bugeaud2004classicalmodularapproachesexponential}, and for some special cases of negative $d$ like $d=-1$ \cite{ko1965diophantine} and $d=-2^k, k \geq 2$ \cite{Coh2007}. Lebesgue-Nagell equations with $-100 \leq d \leq -1$ were studied by Barros \cite{Barros2010}. More recent work \cite{BS2023a,BS2023b} has solved infinite families of Lebesgue--Nagell equations.

The simplest unsolved Lebesgue--Nagell equation, and the focus of this paper, is the Lebesgue--Nagell equation with $d=-2$. By taking a prime factor of $n$, we reduce to studying the equation
\begin{align}\label{eq:main lebesgue nagell eqn}
x^2-2=y^p
\end{align}
for $p$ a prime and $x,y \in \mathbb{Z}$. If $p=2$, then we have $(x+y)(x-y) = 2$, and this has no solutions since $x+y$ and $x-y$ have the same parity. Therefore, we may take $p$ to be an odd prime. 

One reason \eqref{eq:main lebesgue nagell eqn} has resisted resolution is because it has the ``trivial'' solutions $(x, y) = (\pm 1, -1)$ for all odd primes $p$. This makes it difficult to rule out solutions with local approaches or techniques arising from the modularity of elliptic curves. In fact, this is true of any Lebesgue--Nagell equation with $d = - a^2+\epsilon$ where $\epsilon \in \{-1,0,1\}$, for then $(x, y) = (\pm a, \epsilon)$ is a solution for all odd $p$ (see some further discussion in \cite[p. 32]{bugeaud2004classicalmodularapproachesexponential}). 

One conjectures that these trivial solutions to \eqref{eq:main lebesgue nagell eqn} are the only solutions.

\begin{conjecture}\label{conj:solns to main lebesgue nagell eqn}
    Let $p\geq 3$ be an odd prime, and let $x,y$ be integers such that $x^2-2=y^p$. Then $y=-1$.
\end{conjecture}

We say Conjecture \ref{conj:solns to main lebesgue nagell eqn} \emph{holds for $p$} if the only solutions to $x^2-2=y^p$ have $y=-1$. While Conjecture \ref{conj:solns to main lebesgue nagell eqn} remains open (we do not resolve it in this paper!), there has been important partial progress. Siksek \cite[Chapter 15]{Coh2007}, reporting joint work with Bugeaud and Mignotte, described what he called ``a partial attempt at solving'' \eqref{eq:main lebesgue nagell eqn}. We discuss their elegant and insightful work further throughout this paper. Barros \cite{Barros2010} also made progress towards Conjecture \ref{conj:solns to main lebesgue nagell eqn}.

It is possible to solve \eqref{eq:main lebesgue nagell eqn} for small values of $p$ by computation. The following result is due to Barros \cite[Theorem 1.1]{Barros2010}, and we confirm it with a computation in Section \ref{sec:ClassicalApproach}.

\begin{theorem}\label{thm:thue equations small p}
Let $3 \leq p \leq 13$ be a prime. Then Conjecture \ref{conj:solns to main lebesgue nagell eqn} holds for $p$.
\end{theorem}

Solving \eqref{eq:main lebesgue nagell eqn} for a given value of $p$ requires solving related Diophantine equations called \emph{Thue equations} (see Section \ref{sec:ClassicalApproach} for further discussion). It is claimed in \cite[Lemma 15.7.3]{Coh2007} that Conjecture \ref{conj:solns to main lebesgue nagell eqn} holds for $3 \leq p \leq 37$, with the computations performed by GP/PARI (see \cite{PARI2} for the most recent version of the software). However, the default \verb|thue| function in GP/PARI assumes a Generalized Riemann Hypothesis (GRH), and one must pass a flag to the function in order to obtain unconditional results. We were able to use GP/PARI to solve conditionally the relevant Thue equations up to $p=37$, but only up to $p=13$ unconditionally. Thus, it is unclear whether \cite[Lemma 15.7.3]{Coh2007} relies on GRH for its validity.

It is likely that the upper bound in Theorem \ref{thm:thue equations small p} can be extended. The computational bottleneck in solving \eqref{eq:main lebesgue nagell eqn} for larger values of $p$ is the computation of a system of fundamental units of rank $\frac{p-1}{2}$ in a number field of degree $p$ (see \cite{TdW1989} for more information on solving Thue equations). Provided $p$ is not too large, it is likely one can find a system of fundamental units conditionally, assuming GRH, and then use the techniques of \cite{Hanrot2000} to solve the Thue equations unconditionally. The methods of \cite{BH1996} may also be helpful. It would be interesting to see how far Theorem \ref{thm:thue equations small p} can be extended with the aid of significant computing power.

Perhaps the most impressive advance on \eqref{eq:main lebesgue nagell eqn} was made by Chen \cite{Chen_2012}, who used methods related to the modularity of elliptic curves to show that Conjecture \ref{conj:solns to main lebesgue nagell eqn} holds for primes $p$ satisfying certain congruence conditions.

\begin{theorem}[{\cite[Theorem 5]{Chen_2012}}]\label{thm:chens theorem}
    Conjecture \ref{conj:solns to main lebesgue nagell eqn} holds for $p$ if $p\equiv 1, 5, 7, 11 \pmod{24}$.
\end{theorem}

Hence, by Theorem \ref{thm:chens theorem}, one only needs to consider odd primes $p \equiv 13,17,19,23 \pmod{24}$ when attempting to prove Conjecture \ref{conj:solns to main lebesgue nagell eqn}.

In order to reduce the proof of Conjecture \ref{conj:solns to main lebesgue nagell eqn} to a finite computation, even in principle, one needs to know that if $p$ is sufficiently large, then the only solutions to \eqref{eq:main lebesgue nagell eqn} have $y=-1$. This is provided by the following theorem.

\begin{theorem}\label{thm:only triv solns for p big}
Conjecture \ref{conj:solns to main lebesgue nagell eqn} holds for $p > 911$. That is, if $x,y \in \mathbb{Z}$ with $x^2-2=y^p$ and $p>911$, then $y=-1$.
\end{theorem}

Barros \cite[Theorem 1.1]{Barros2010} obtained a version of Theorem \ref{thm:only triv solns for p big} with $911$ replaced by $4111$.

The technique behind the proof of Theorem \ref{thm:only triv solns for p big} is a careful application of linear forms in two logarithms (see Section \ref{sec:Linear forms in logarithms}). It is claimed without proof in \cite[p. 520]{Coh2007} that, using results on linear forms in two logarithms from \cite{Laurent1995}, one can show Conjecture \ref{conj:solns to main lebesgue nagell eqn} holds for $p \geq 1237$. This was recently confirmed by Bennett, Pink, and Vukusic \cite{BPV2024}. However, in order to reach $p \leq 1237$, the authors of \cite{BPV2024} had to use the more recent bounds on linear forms in two logarithms in \cite{Laurent2008}, and also had to use Theorem \ref{thm:chens theorem}.

By Theorem \ref{thm:only triv solns for p big}, it suffices to consider $p\leq 911$ when making further attempts at Conjecture \ref{conj:solns to main lebesgue nagell eqn}. There are ``only'' 84 primes $17 \leq p \leq 911$ with $p \equiv 13,17,19,23 \pmod{24}$, so Theorems \ref{thm:thue equations small p}, \ref{thm:chens theorem}, and \ref{thm:only triv solns for p big} taken together reduce Conjecture \ref{conj:solns to main lebesgue nagell eqn} to 84 cases.

While we cannot show Conjecture \ref{conj:solns to main lebesgue nagell eqn} holds for all $p \leq 911$, we can show that any counterexample $y$ to Conjecture \ref{conj:solns to main lebesgue nagell eqn} must be large.

\begin{theorem} \label{thm:lower bound on y}
Let $p$ be an odd prime, and let $x,y$ be integers with $x^2-2=y^p$. If $y \neq -1$, then $y > 10^{1000}$.
\end{theorem}

Additional computations could increase the lower bound in Theorem \ref{thm:lower bound on y} even further. Theorem \ref{thm:lower bound on y} refines \cite[Corollary 25]{Chen_2012}, which states a lower bound $y>10^{102}$ (note, however, that \cite[Corollary 25]{Chen_2012} relies on \cite[Lemma 15.7.3]{Coh2007}, which we discussed above).

In addition to the highlighted results above, we prove a number of other results in this work relating to solutions of \eqref{eq:main lebesgue nagell eqn}. We refer the reader to the outline of the paper below (subsection \ref{subsec:outline of paper}) for a brief, general overview of these results, and to the relevant sections for particular information. It is hoped that these ancillary results will be useful in further investigations of \eqref{eq:main lebesgue nagell eqn}. For readers interested only in the proofs of the theorems mentioned here in the introduction, it is recommended to read the brief Section \ref{sec:ClassicalApproach}, and then proceed to the technical heart of the paper in Sections \ref{sec:Linear forms in logarithms} and \ref{sec:small y} (following references to results in other sections as desired).

As described above, it is conjectured that all solutions to \eqref{eq:main lebesgue nagell eqn} are ``trivial.'' It is helpful throughout our work to refer to ``nontrivial'' solutions to \eqref{eq:main lebesgue nagell eqn}. We encapsulate these notions in the following definition.

\begin{definition}\label{defn:triv and nontriv solns}
    Let $x,y \in \mathbb{Z}$ and $p\geq 3$ be an odd prime such that $x^2-2=y^p$. We say the solution $(x,y)$ is \emph{trivial} if $y = -1$, and \emph{nontrivial} if $y \neq -1$.
\end{definition}

Observe that a nontrivial solution in the sense of Definition \ref{defn:triv and nontriv solns} has $y$ positive.

\subsection{Notation}

We use the following notation throughout the paper:

\begin{itemize}
    \item $v_p(n)$: $p$-adic valuation of $n$.

    \item $\sgn(x)$: $1$ if $x > 0$ and $-1$ if $x < 0$.

    \item $\F_p$: finite field of prime order $p$.

    \item $\aff(\F_p)$: The group of affine functions $x \mapsto ax+b$ for $a \in \F_p^\times, b \in \F_p$ under function composition. This is a semidirect product $\F_p \rtimes \F_p^\times$.

    \item $\aut(G)$: The automorphism group of a group $G$.

    \item $D_{2n}$: The dihedral group of order $2n$.

    \item $\# S$: cardinality of a finite set $S$.

    \item For an algebraic number $\alpha$ whose minimal polynomial over $\Z$ is $a x^d + \dots$ and conjugates are $\alpha^{(1)}, \dots, \alpha^{(d)}$, the height of $\alpha$ is defined by
\begin{align*}
    h(\alpha) = \frac{1}{d}\left(\log \abs{a} + \sum_{i=1}^d \log \max(1, \abs{\alpha^{(i)}})\right).
\end{align*}

    \item $\left(\frac{n}{p}\right)$: the Legendre symbol for $n, p \in \Z$, $p$ an odd prime.

    \item For a positive integer $n$, we write $\text{rad}(n)=\prod_{p \mid n} p$.

    \item $\proj^1(R)$: the projective line over a ring $R$.

    \item For an elliptic curve $E/\mathbb{Q}$ and a prime $\ell$, we write $a_\ell(E)$ for $\ell+1-\#E(\F_\ell)$.

\end{itemize}

\subsection{Outline of Paper}\label{subsec:outline of paper}

In Section \ref{sec:elementary observations}, we begin by establishing several necessary conditions that any nontrivial solution to equation \eqref{eq:main lebesgue nagell eqn} must satisfy. These elementary results include congruence conditions on $x$ and $y$, such as the fact that $x$ and $y$ must both be odd.

In Section \ref{sec:ClassicalApproach}, following \cite[Chapter 15]{Coh2007}, we factor $x^2-2$ in the ring $\Z[\sqrt{2}]$ and demonstrate that nontrivial solutions of \eqref{eq:main lebesgue nagell eqn} imply the existence of nontrivial solutions to certain Diophantine equations called \emph{Thue equations}. These Thue equations are indexed by the odd prime $p$ of \eqref{eq:main lebesgue nagell eqn}, and also by an integer $r$ with $|r| \leq \frac{p-1}{2}$. We give the (easy) proof of Theorem \ref{thm:thue equations small p} at the end of this section. 

In Section \ref{sec:Galois Theory}, we study the polynomials corresponding to these Thue equations in some depth. We find an explicit formula for the roots and use this to determine the Galois group and discriminant of the polynomials. We also show these polynomials are irreducible over $\Q[\sqrt{2}]$ and have exactly one real root.

Next, in the long Section \ref{sec:Linear forms in logarithms}, we give a detailed account of how to apply linear forms in logarithms to obtain the bound on $p$ in Theorem \ref{thm:only triv solns for p big} for all sufficiently large values of $y$. This section is split into several different subsections, and the argument proceeds in various stages. First, we obtain a somewhat crude, preliminary upper bound on a prime $p$ such that \eqref{eq:main lebesgue nagell eqn} admits a nontrivial solution (Theorem \ref{thm:initial bound}). With $p$ suitably bounded, we use the argument of \cite[Proposition 15.7.1]{Coh2007} to show that $r=\pm 1$. With $r$ now restricted, we apply linear forms in logarithms again to reduce the upper bound on $p$ further (Theorem \ref{thm:improved bound thm 2}). Once the upper bound on $p$ has been reduced yet again, we employ linear forms in logarithms for the third time, now with a delicate choice of parameters and a careful analysis. In order to reduce the upper bound on $p$ as far as possible, we need to assume that any nontrivial solution to \eqref{eq:main lebesgue nagell eqn} has $y$ large.

In Section \ref{sec:small y}, we complete the proof of Theorem \ref{thm:only triv solns for p big} by a continued fraction computation allowing us to rule out small solutions to the Thue equation. We also use the same argument for the remaining values of $p$ to prove Theorem \ref{thm:lower bound on y}.

Next, in Section \ref{sec:other observations}, we make various elementary observations about any nontrivial solutions to equation \eqref{eq:main lebesgue nagell eqn} and the corresponding Thue equations. These observations rely on the fact that $r = \pm 1$.

In Section \ref{sec:simple modular observations}, we examine what information there is to be gained from the modular method applied to an elliptic curve associated to solutions of \eqref{eq:main lebesgue nagell eqn}. As a weaker conjecture than Conjecture \ref{conj:solns to main lebesgue nagell eqn}, we conjecture that all nontrivial solutions to \eqref{eq:main lebesgue nagell eqn} are ``locally trivial,'' in that $x \equiv \pm 1 \pmod{p}$ and $y \equiv -1 \pmod{p}$ (Conjecture \ref{conj:triv solns mod p}). We make very modest progress towards this latter conjecture.

In Section \ref{sec:newform coefficients}, we examine the newforms of level 128 corresponding to solutions of \eqref{eq:main lebesgue nagell eqn} and give a proof sketch of how one can obtain explicit formulas for the coefficients of these newforms. We speculate that studying the explicit formulas for the coefficients of these newforms might allow for a more effective deployment of modularity techniques.

Finally, Section \ref{sec:solutions mod n} contains an examination of the solutions to the Thue equations modulo an integer $n$. We determine, in most cases, exactly how many solutions there are (Theorem \ref{thm:thue equations mod n}).

\subsection{Reproducibility of Computations}\label{subsec:reproducibility}

The work in this paper relies on some computer calculation. Our code and associated data can be found at the GitHub repository \cite{Katz_Lebesgue-Nagell_Helper_Code}. All scripts found in this repository were successfully executed using one of Mathematica version 13.1.0.0, Python version 3.10.6, or Sage version 10.6 (using Python 3.11.12 and GP/PARI version 2.17.2). The operating system used was Windows 11, with Sage being accessed from Ubuntu 24.04.2 using the Windows Subsystem for Linux.

\section{Elementary observations} \label{sec:elementary observations}

In this section, we establish various conditions that any nontrivial solution to \eqref{eq:main lebesgue nagell eqn} must satisfy. The proofs are elementary, and do not invoke ``heavy'' tools like linear forms in logarithms or modularity. However, some results from this section will find application later on when we do use such tools.

Variants of the following lemma appear in many works on exponential Diophantine equations. For example, see \cite[page 16, exercise 3.2]{Schoof2008}.

\begin{lemma}\label{lemma1}
Let $a,y$ be distinct, coprime integers, and let $p\geq 3$ be a prime. The following are true:
\begin{enumerate}

	\item $\gcd\left(y-a, \frac{y^p-a^p}{y-a}\right)$ is $1$ or $p$.\label{item:lemma1.1}
	
	\item $p \mid (y-a)$ if and only if $p \mid \frac{y^p-a^p}{y-a}$. In this case, $\frac{y^p-a^p}{y-a} \equiv p \pmod{p^2}$ if $p \nmid a$, and $\frac{y^p-a^p}{y-a} \equiv 0 \pmod{p^2}$ if $p \mid a$.\label{item:lemma1.2}
	
	\item If $q \neq p$ is prime and $q \mid \frac{y^p-a^p}{y-a}$, then $q \equiv 1 \pmod{p}$.\label{item:lemma1.3}

\end{enumerate}
\end{lemma}

\begin{proof}
Note that
\begin{align*}
y^p-a^p &= (y-a)\left(\sum_{i=0}^{p-1} a^{p-1-i} y^i\right).
\end{align*}
Let $d \in \Z$ be a common divisor of $y-a$ and $\frac{y^p-a^p}{y-a}$. Since $d \mid (y-a)$, we have $y \equiv a \pmod{d}$, so\[
\sum_{i=0}^{p-1} a^{p-1-i} y^i \equiv \sum_{i=0}^{p-1} a^{p-1-i} a^i = p a^{p-1} \pmod{d}.
\]
Thus, $d \mid p a^{p-1}$. Observe that $\text{gcd}(d,a)=1$ since $d \mid(y-a)$ and $a,y$ are coprime. Thus, $d \mid p$, proving \eqref{item:lemma1.1}. 

Since $y \equiv y^p \pmod{p}$ and $ a \equiv a^p \pmod{p}$, it follows $p \mid (y-a)$ if and only if $p \mid (y^p-a^p)$. Thus, if $p \mid \frac{y^p-a^p}{y-a}$, then $p \mid (y-a)$. Now, if $p \mid (y-a)$, then
\begin{align*}
\frac{y^p-a^p}{y-a} &= \frac{(y-a+a)^p-a^p}{y-a}= \sum_{k=1}^p \binom{p}{k} a^{p-k} (y-a)^{k-1}\equiv p a^{p-1} \pmod{p^2},
\end{align*}
because any term except $k=1$ vanishes mod $p^2$ since $p \geq 3$. By Fermat's little theorem, if $p \nmid a$ then $a^{p-1} \equiv 1 \pmod{p}$, so $p a^{p-1} \equiv p \pmod{p^2}$. If $p \mid a$ then $p a^{p-1} \equiv 0 \pmod{p^2}$. This proves \eqref{item:lemma1.2}.

Let $q \neq p$ be a prime dividing $\frac{y^p-a^p}{y-a}$. Then $q$ divides $y^p-a^p$, so $y^p \equiv a^p \pmod{q}$. If $q \not\equiv 1 \pmod{p}$, then $p \nmid \phi(q)$, so the map $x \mapsto x^p$ is injective mod $q$, and hence $y \equiv a \pmod{q}$. Thus $q \mid y-a$, contradicting \eqref{item:lemma1.1}. This establishes \eqref{item:lemma1.3}.
\end{proof}

For the remainder of the results in this section, we assume $x, y, p \in \Z$ form a nontrivial solution to \eqref{eq:main lebesgue nagell eqn}, with $p \geq 3$ a prime and $y \neq -1$.

\begin{theorem}\label{thm:elementary 1}
    The integers $x$ and $y$ are both odd, $y \equiv 7 \pmod{8}$, and every prime that divides $y$ is $\equiv \pm 1 \pmod{8}$.
\end{theorem}

\begin{proof}
    Clearly $x^2$ and $y^p$ have the same parity, and so $x$ and $y$ must, too. If $x$ and $y$ are even, then $x^2-2 \equiv 2 \pmod{4}$ and $y^p \equiv 0 \pmod{4}$, which is a contradiction. Thus, $x$ and $y$ are odd. This implies $x^2 \equiv 1 \pmod{8}$, so $y^p = x^2-2 \equiv -1 \pmod{8}$. Since $y^{p-1} \equiv 1 \pmod{8}$, this implies $y \equiv -1 \equiv 7 \pmod{8}$. Now, for any prime $q$ dividing $y$, $q$ divides $x^2-2$, so $2$ is a quadratic residue mod $q$. By quadratic reciprocity, since $q \neq 2$, this implies $q \equiv \pm 1 \pmod{8}$.
\end{proof}

\begin{theorem}\label{thm:elementary 2}
We have two cases:
    \begin{enumerate}

        \item If $y \not\equiv -1 \pmod{p}$, then $x \not\equiv \pm 1 \pmod{p}$.

        \item If $y \equiv -1 \pmod{p}$, then $x \equiv \pm 1 \pmod{p^2}$, and, more precisely, $v_p((x-1)(x+1)) = v_p(y+1)+1$.
        
    \end{enumerate}
\end{theorem}

\begin{proof}
By \eqref{eq:main lebesgue nagell eqn} we have $x^2-1=y^p+1$, and factoring gives
\begin{align}
(x+1)(x-1) &= (y+1)(1-y+y^2-\dots+y^{p-1}) \label{eq1}.
\end{align}
We apply Lemma \ref{lemma1} with $a = -1$. If $p \nmid (y+1)$, then the factors on the right of \eqref{eq1} are both not divisible by $p$. If $p \mid (y+1)$, then both factors are divisible by $p$, and $\frac{y^p+1}{y+1} \equiv p \pmod{p^2}$. In the second case, $v_p\left(\frac{y^p+1}{y+1}\right) = 1$, so $v_p((x-1)(x+1)) = v_p\left((y+1)\left(\frac{y^p+1}{y+1}\right)\right) = v_p(y+1)+1$.
\end{proof}

\begin{remark}\label{rmk:elementary 1}
    Theorem \ref{thm:elementary 2} has the interesting consequence that either $x \not\equiv \pm 1 \pmod{p}$, or $x \equiv \pm 1 \pmod{p^2}$.
\end{remark}

\begin{theorem}\label{thm:elementary 3}
    If $3 \mid x$, then $y \equiv 7 \pmod{24}$. If $3 \nmid x$, then $y \equiv 23 \pmod{24}$. Also, $3 \mid (y-1)$ if and only if $3 \mid x$.
\end{theorem}

\begin{proof}
    Note $y \equiv y^p = x^2-2 \pmod{3}$, which is $\equiv 1$ if $3 \mid x$, and $\equiv 2$ if $3 \nmid x$. Combining with $y \equiv 7 \pmod{8}$ from Theorem \ref{thm:elementary 1}, we have $y \equiv 7, 23 \pmod{24}$ if $3 \mid x, 3 \nmid x$, respectively. Then $y-1$ is congruent to $6$ or $22 \pmod{24}$, respectively, so $3 \mid (y-1)$ if and only if $3 \mid x$.
\end{proof}

In the next two theorems, we rearrange \eqref{eq:main lebesgue nagell eqn} as $x^2 - 3 = y^p - 1$, so
\begin{align}
x^2-3 &= (y-1)(1+y+y^2+\dots+y^{p-1}) \label{eq2}.
\end{align}
\begin{lemma}\label{thm:elementary lemma 2}
    Any prime $q$ that divides one side of \eqref{eq2} is either $2, 3$, or $\equiv \pm 1 \pmod{12}$.
\end{lemma}
\begin{proof}
Any prime $q \geq 5$ that divides $x^2-3$ must have $\left(\frac{3}{q}\right) = 1$. By quadratic reciprocity, if $q\equiv 1 \pmod{4}$, then $q\equiv 1 \pmod{3}$, and if $q \equiv -1 \pmod{4}$, then $q\equiv -1 \pmod{3}$. Hence, $q \equiv \pm 1 \pmod{12}$.
\end{proof}

\begin{theorem}\label{thm:elementary 4}
    We have $v_2(y-1) = 1, v_3(y-1) = 0$ or $1$, and every other prime that divides $y-1$ is $\equiv \pm 1 \pmod{12}$.
\end{theorem}

\begin{proof}
    By Theorem \ref{thm:elementary 1}, $y-1 \equiv 6 \pmod{8}$, so $v_2(y-1) = 1$. If $3 \mid (y-1)$, then $3 \mid x$ by Theorem \ref{thm:elementary 3}, so $x^2-3 \equiv 6 \pmod{9}$, so $9 \nmid (y-1)$. Hence, $v_3(y-1) = 0$ or $1$. By Lemma \ref{thm:elementary lemma 2}, every other prime that divides $y-1$ is $\equiv\pm 1 \pmod{12}$.
\end{proof}

\begin{theorem}\label{thm:elementary 5}
    We have the following:
    \begin{enumerate}
	
	\item Every prime factor of $\frac{y^p-1}{y-1}$ is $\equiv\pm 1 \pmod{12}$\label{item:elementary 5.1}.
    
        \item Every prime factor of $\frac{y^p-1}{y-1}$ is $p$  or $\equiv 1 \pmod{p}$. Either $p \nmid (y-1)$ and $\frac{y^p-1}{y-1} \equiv 1 \pmod{p}$, or $p \mid (y-1)$ and $\frac{y^p-1}{y-1} \equiv p \pmod{p^2}$. The latter case cannot happen unless $p \equiv \pm 1 \pmod{12}$.\label{item:elementary 5.2}

        \item $\frac{y^p-1}{y-1} \equiv 1 \pmod{24}$\label{item:elementary 5.3}.
	
	\item If $p \equiv 2 \pmod{3}$ or $p=3$, then $3 \nmid x$.\label{item:elementary 5.4}
    \end{enumerate}
\end{theorem}
\begin{proof}
Since $y$ is odd by Theorem \ref{thm:elementary 1}, $1+y+y^2+\dots + y^{p-1} \equiv p \equiv 1 \pmod{2}$, so $2 \nmid \frac{y^p-1}{y-1}$. If $3 \mid \frac{y^p-1}{y-1}$, then $3 \mid (y^p-1)$, so $y-1 \equiv y^p-1 \equiv 0 \pmod{3}$, so $3 \mid (y-1)$. By \eqref{eq2}, we then have $9 \mid x^2-3$, which is impossible. Thus, $2$ and $3$ do not divide $\frac{y^p-1}{y-1}$. By Lemma \ref{thm:elementary lemma 2}, $\frac{y^p-1}{y-1}$ is only divisible by primes that are $\equiv\pm 1 \pmod{12}$. This proves \eqref{item:elementary 5.1}.

By Lemma \ref{lemma1} with $a=1$, every prime factor of $\frac{y^p-1}{y-1}$ is $p$ or $\equiv 1 \pmod{p}$. If $p \nmid (y-1)$, then $\frac{y^p-1}{y-1} \equiv \frac{y-1}{y-1} \equiv 1 \pmod{p}$. If $p \mid (y-1)$, then $\frac{y^p-1}{y-1} \equiv p \pmod{p^2}$ by Lemma \ref{lemma1}, so $p \mid \frac{y^p-1}{y-1}$, so $p \equiv \pm 1 \pmod{12}$ by statement (1). This proves \eqref{item:elementary 5.2}.

Since $y \equiv -1 \pmod{8}$ by Theorem \ref{thm:elementary 1}, $1+y+y^2+\dots + y^{p-1} \equiv 1 \pmod{8}$, so to prove \eqref{item:elementary 5.3} it suffices to show $\frac{y^p-1}{y-1} \equiv 1 \pmod{3}$. If $3 \nmid (y-1)$, then $\frac{y^p-1}{y-1} \equiv \frac{y-1}{y-1} = 1 \pmod{3}$ and we are done. If $3 \mid (y-1)$, then $3 \mid x$ and $y \equiv 7 \pmod{24}$ by Theorem \ref{thm:elementary 3}, so $1 + y + \dots + y^{p-1} \equiv 1 + 7 + 1 + 7 + \dots + 1 = \frac{p+1}{2} + 7 \frac{p-1}{2} = 4p-3 \pmod{24}$. Since $\frac{y^p-1}{y-1}$ is only divisible by primes that are $\equiv\pm 1 \pmod{12}$ by \eqref{item:elementary 5.1}, it is itself $\equiv\pm 1 \pmod{12}$, so $4p-3 \equiv \pm 1 \pmod{12}$. We cannot have $4p-3\equiv -1 \pmod{12}$, so $\frac{y^p-1}{y-1} \equiv 1 \pmod{12}$. This concludes the proof of \eqref{item:elementary 5.3}. 

In the above paragraph, we showed that if $3 \mid x$, then $4p-3 \equiv 1 \pmod{12}$, so $p \equiv 1 \pmod{3}$, proving \eqref{item:elementary 5.4}.
\end{proof}

\section{Reduction to Thue equations}{\label{sec:ClassicalApproach}}

The following result, contained in \cite[p. 518]{Coh2007}, is of great importance in our work. (See also \cite[Theorem 2.1]{Barros2010}.)
\begin{theorem}\label{thm:reduction to thue equation}
Let $x, y, p \in \Z$ with $p \geq 3$ a prime such that $x^2-2=y^p$. Then there exist $a, b, r \in \Z$ with $\abs{r} \leq \frac{p-1}{2}$ such that
\begin{align}
x+\sqrt{2} = (1+\sqrt{2})^r (a+b\sqrt{2})^p. \label{eq5}
\end{align}
Hence, the integers $a$ and $b$ form a solution to the \emph{Thue equation}
\begin{align}
\frac{1}{2\sqrt{2}} \left((1+\sqrt{2})^r (a+b\sqrt{2})^p - (1-\sqrt{2})^r (a-b\sqrt{2})^p\right) = 1. \label{eq3}
\end{align}
\end{theorem}
\begin{proof}
We work in the ring of integers $\mathbb{Z}[\sqrt{2}]$. Factoring \eqref{eq:main lebesgue nagell eqn} over this ring yields
\begin{align*}
(x+\sqrt{2})(x-\sqrt{2}) &= y^p.
\end{align*}
Fortunately, $\Z[\sqrt{2}]$ is a unique factorization domain. If a prime $\pi$ of $\mathbb{Z}[\sqrt{2}]$ divides $x+\sqrt{2}$ and $x-\sqrt{2}$, then it divides their difference $2\sqrt{2}$, so it must be $\sqrt{2}$, which is the prime of $\mathbb{Z}[\sqrt{2}]$ above $2$. Hence $\sqrt{2}$ divides $x$, so for some $a, b \in \Z$ we have $x = (a+b\sqrt{2})\sqrt{2} = 2b+a\sqrt{2}$, so $a=0$ and $x = 2b$. Thus $x$ is even, but this contradicts the fact that $x$ is odd by Theorem \ref{thm:elementary 1}. Thus, the two factors $x+\sqrt{2}$ and $x-\sqrt{2}$ are relatively prime, so they are both a unit times a $p$th power. The units in $\Z[\sqrt{2}]$ are generated by $-1$ and $1+\sqrt{2}$, so by folding the $-1$ into the $p$th power we have \eqref{eq5}. By folding $p$th powers of $1+\sqrt{2}$ into the $(a+b\sqrt{2})^p$, we may also assume $-\frac{p-1}{2} \leq r \leq \frac{p-1}{2}$. 

We obtain \eqref{eq3} by subtracting from \eqref{eq5} its conjugate and dividing by $2\sqrt{2}$.
\end{proof}

Later, in Theorem \ref{thm: r is 1}, we show that we must have $r = \pm 1$. Notice that if $x+\sqrt{2} = (1+\sqrt{2})^{-1} (a+b\sqrt{2})^p$, then
\begin{align*}
x-\sqrt{2} &= (1-\sqrt{2})^{-1} (a-b\sqrt{2})^p\\
&= -(1+\sqrt{2})(a-b\sqrt{2})^p
\end{align*}
By making the substitution $x \mapsto -x, b \mapsto -b$, we obtain again $x+\sqrt{2} = (1+\sqrt{2})(a+b\sqrt{2})^p$. Thus, we may in fact assume $r=1$, although we can no longer control the sign of $x$ by doing so. This implies that, once we have established Theorem \ref{thm: r is 1}, we only need to solve the single Thue equation
\begin{align}
1 &= \frac{1}{2\sqrt{2}} \left((1+\sqrt{2})(a+b\sqrt{2})^p - (1-\sqrt{2}) (a-b\sqrt{2})^p\right)  \label{eq4}\\
&= \sum_{k=0}^p \binom{p}{k} 2^{\floor{\frac{k}{2}}} a^{p-k} b^k \nonumber
\end{align}
for each value of $p$. This equation has the ``trivial'' solution $(a, b) = (1, 0)$ for all $p$, corresponding to the trivial solutions of \eqref{eq:main lebesgue nagell eqn}.

\begin{proof}[Proof of Theorem \ref{thm:thue equations small p}]
    Let $p\geq 3$ be a prime. By Theorem \ref{thm:reduction to thue equation}, any solution $x^2-2=y^p$ to \eqref{eq:main lebesgue nagell eqn} gives rise to a solution to a Thue equation \eqref{eq3} with $|r| \leq \frac{p-1}{2}$. We use GP/PARI's built-in Thue equation solver (accessed through Sage \cite{sagemath}) to solve unconditionally all the Thue equations when $p \leq 13$. The Thue equations only have solutions when $r = \pm 1$, and in each case the solutions correspond to $y=-1$. See the file \verb|theorem_1_2.sage| of \cite{Katz_Lebesgue-Nagell_Helper_Code} for the code we used.
\end{proof}

\section{Observations from Galois theory} \label{sec:Galois Theory}

In this section, we make several observations about the roots and field extensions corresponding to the Thue equations in Theorem \ref{thm:reduction to thue equation}. 

Let $f_{r,p}(x) = \frac{1}{2\sqrt{2}} \left((1+\sqrt{2})^r (x+\sqrt{2})^p - (1-\sqrt{2})^r (x-\sqrt{2})^p\right)\in \mathbb{Z}[x]$ be the polynomial corresponding to the Thue equation \eqref{eq3}. We assume throughout this section that $p$ is an odd prime and $p \nmid r$.

\begin{theorem}\label{thm:formula for roots}
    The polynomial $f_{r,p}$ has $p$ distinct roots $\rho_0, \dots, \rho_{p-1}$, given by the formula
    \begin{align}
        \rho_i &= \sqrt{2} + \frac{2\sqrt{2}}{(-1)^r (1+\sqrt{2})^{-2r/p} \zeta_p^i - 1}. \label{eq503}
    \end{align}
    where $\zeta_p = e^{2\pi i/p}$. The only real root is $\theta = \theta_{r,p} = \rho_0$.
\end{theorem}
\begin{proof}
We have that $f_{r,p}(x) = 0$ if and only if $(1+\sqrt{2})^r(x+\sqrt{2})^p = (1-\sqrt{2})^r(x-\sqrt{2})^p$. We put the $p$th powers on one side of the equation, and the $r$th powers on the other. Simplifying slightly, we obtain
\begin{align*}
    \left(1+\frac{2\sqrt{2}}{x-\sqrt{2}}\right)^p = (-1)^r (1+\sqrt{2})^{-2r}.
\end{align*}
Taking $p$th roots and rearranging, we find the roots are given by
\begin{align}
    \rho_i = \sqrt{2} + \frac{2\sqrt{2}}{(-1)^r(1+\sqrt{2})^{-2r/p} \zeta_p^i - 1} \label{eq501}
\end{align}
for $0 \leq i \leq p-1$.

Now, $\zeta_p^i$ is determined by $x$, since solving for $\zeta_p^i$ in the equation above we have
\begin{align}
\zeta_p^i = (-1)^r (1+\sqrt{2})^{2r/p} \left(1+\frac{2\sqrt{2}}{\rho_i-\sqrt{2}}\right) \label{eq502}
\end{align}
From \eqref{eq501} and \eqref{eq502} we see that $\rho_i$ is real if and only if $\zeta_p^i$ is real, that is, if $\zeta_p^i = 1$. Therefore, there is precisely one real root, which we call $\theta$. Since $\zeta_p^i \neq \zeta_p^j$ for $i \not\equiv j \pmod{p}$, from \eqref{eq502} we deduce $\rho_i \neq \rho_j$ for $i \not\equiv j \pmod{p}$. Thus, $\theta = \rho_0, \rho_1, \rho_2, \dots, \rho_{p-1}$ are $p$ distinct roots of $f_{r,p}$, a degree $p$ polynomial, so the $\rho_i$ are all of the roots of $f_{r,p}$, each occurring with a multiplicity of one.
\end{proof}

\begin{remark}
    We think of the index $i$ in $\rho_i$ as a residue class mod $p$. See, particularly, the proof of Theorem \ref{thm:gal gp of L over Q} below.
\end{remark}

Theorem \ref{thm:formula for roots} has the following immediate corollary.

\begin{corollary}\label{corollary1}
The splitting field $L$ of $f_{r,p}$ is a radical extension of $\Q$ lying in the field $S = \Q((1+\sqrt{2})^{1/p}, \zeta_p)$.
\end{corollary}

Note that the field $S$ in Corollary \ref{corollary1} is a Galois extension of $\Q$.

It is helpful for our later work on linear forms in logarithms to record some information about the location of the real root $\theta$ of $f_{r,p}$.

\begin{proposition} \label{prop:real roots}
$ $
\begin{enumerate}
	\item If $r > 0$ and $r$ is even, then $\theta < -\sqrt{2}$.
	
	\item If $r > 0$ and $r$ is odd, then $-\sqrt{2} < \theta < 0$.
	
	\item If $r < 0$ and $r$ is even, then $\theta > \sqrt{2}$.
	
	\item If $r < 0$ and $r$ is odd, then $0 < \theta < \sqrt{2}$.
\end{enumerate}
\end{proposition}

\begin{proof}
This follows immediately from the expression
\begin{align}
\theta &= \sqrt{2} + \frac{2\sqrt{2}}{(-1)^r (1+\sqrt{2})^{-2r/p} - 1}. \label{eq504}
\end{align}
\end{proof}

With the above as preamble, our main task now is to determine the Galois group of $L$, the splitting field of $f_{r,p}$. It is somewhat difficult to access $L$ directly, and it will turn out to be more convenient to study $L$ by studying the larger field $S$.

The expression \eqref{thm:formula for roots} for the roots can also be written as
\begin{align}
\rho_i &= \sqrt{2} \frac{(1-\sqrt{2})^{r/p} \zeta_p^i + (1+\sqrt{2})^{r/p}}{(1-\sqrt{2})^{r/p} \zeta_p^i - (1+\sqrt{2})^{r/p}}. \label{eq505}
\end{align}
The expression in \eqref{eq505} will be useful for determining how the Galois group acts on the roots.

\begin{lemma}\label{lemma4.4}
The minimal polynomial of $(1+\sqrt{2})^{r/p}$ over $\Q(\sqrt{2})$ is $x^p-(1+\sqrt{2})^r$ if $p \nmid r$.
\end{lemma}

\begin{proof}
By a well-known theorem in Galois theory (see, for example, \cite[Proposition 4.2.6]{Cox2012}), it suffices to show that $(1+\sqrt{2})^r$ is not a $p$th power in $\Q(\sqrt{2})$. So, assume by way of contradiction that $(1+\sqrt{2})^r = \alpha^p$ for some $\alpha \in \Q(\sqrt{2})$. Write $\alpha = \frac{a+b\sqrt{2}}{d}$, where $a, b, d \in \Z$. Then we have
\begin{align*}
(a+b\sqrt{2})^p = (1+\sqrt{2})^r d^p
\end{align*}
We use the fact that $\Z[\sqrt{2}]$ has unique factorization and the fundamental unit is $1+\sqrt{2}$ to write
\begin{align*}
a+b\sqrt{2} &= \epsilon_1 (1+\sqrt{2})^s \prod_k \pi_k, \ \ \ \ \ \ d = \epsilon_2 (1+\sqrt{2})^t \prod_j \kappa_j,
\end{align*}
where $\epsilon_1, \epsilon_2 \in \{\pm 1\}$ and $\pi_k$ and $\kappa_j$ are primes in $\Z[\sqrt{2}]$. Then we have
\begin{align*}
\epsilon_1^p (1+\sqrt{2})^{sp} \prod_k \pi_k^p &= \epsilon_2^p (1+\sqrt{2})^{r+tp} \prod_j \kappa_j^p.
\end{align*}
By unique factorization, $\prod_k \pi_k^p = \prod_j \kappa_j^p$, so dividing out we have
\begin{align*}
(1+\sqrt{2})^{r+p(t-s)} = \pm 1
\end{align*}
But since $p \nmid r$, $r+p(t-s) \neq 0$, so some nonzero power of $1+\sqrt{2}$ equals $\pm 1$, a contradiction.
\end{proof}

Lemma \ref{lemma4.4} shows $(1+\sqrt{2})^{r/p}$ has degree $p$ over $\mathbb{Q}(\sqrt{2})$, so we have the following corollary.

\begin{corollary} \label{corollary2}
    $[\Q((1+\sqrt{2})^{1/p}):\Q] = 2p$
\end{corollary}

\begin{theorem}$ $
\begin{enumerate}
    \item $\Q(\sqrt{2}, \theta) = \Q((1+\sqrt{2})^{1/p})$
    \item $f_{r,p}$ is irreducible over $\Q(\sqrt{2})$, hence also over $\Q$.
\end{enumerate}
\end{theorem}
\begin{proof}
Clearly $\sqrt{2} \in \Q((1+\sqrt{2})^{1/p})$. From \eqref{eq504} it is obvious that $\theta \in \Q((1+\sqrt{2})^{1/p})$, and also that $(1+\sqrt{2})^{2r/p} \in \Q(\sqrt{2}, \theta)$. Since $2r$ is coprime to $p$, by raising this to the power of the multiplicative inverse of $2r$ mod $p$, we see that $(1+\sqrt{2})^{1/p} \in \Q(\sqrt{2}, \theta)$. Thus, $\Q(\sqrt{2}, \theta) = \Q((1+\sqrt{2})^{1/p})$. 

Since $\theta$ is a root of the degree $p$ polynomial $f_{r,p}$ over $\Q(\sqrt{2})$, and $[\Q(\sqrt{2}, \theta):\Q(\sqrt{2})] = [\Q((1+\sqrt{2})^{1/p}):\Q(\sqrt{2})] = p$ by Lemma \ref{lemma4.4}, it follows that $f_{r,p}$ is irreducible over $\Q(\sqrt{2})$.
\end{proof}

\begin{corollary} \label{corollary3}
    Let $L$ and $S$ be defined as in Corollary \ref{corollary1}. Then $L(\sqrt{2}) = S$.
\end{corollary}
\begin{proof}
Since $L$ is the splitting field of $f_{r,p}$ over $\Q$, $L(\sqrt{2})$ is the splitting field of $(x^2-2)f_{r,p}(x)$ over $\Q$, hence it is Galois over $\Q$. Clearly, $S$ is the Galois closure of $\Q((1+\sqrt{2})^{1/p})$. Since $\Q(\theta) \subseteq L$ and $\Q((1+\sqrt{2})^{1/p}) = \Q(\sqrt{2}, \theta) \subseteq L(\sqrt{2})$, we see $S \subseteq L(\sqrt{2})$. But from Corollary \ref{corollary1}, $L \subseteq S$, and $\sqrt{2} \in S$, so also $L(\sqrt{2}) \subseteq S$.
\end{proof}

\begin{theorem} \label{thm:degree of S}
    $[S:\Q] = 2p(p-1)$
\end{theorem}
\begin{proof}
It is a well-known fact from algebraic number theory that the cyclotomic field $\Q(\zeta_p)$ has cyclic Galois group of order $p-1$ over $\Q$ (see \cite[p. 238]{Cox2012}), hence it has a unique quadratic subfield, which is $\Q(\sqrt{(-1)^{(p-1)/2} p})$ (see, for example, \cite[p. 40, exercise 8]{NumberFieldsMarcus}). In particular, $\sqrt{2} \not\in \Q(\zeta_p)$, so $[\Q(\zeta_p, \sqrt{2}): \Q(\zeta_p)] = 2$. Because $[\Q(\zeta_p):\Q] = p-1$, we deduce by the tower theorem that $[\Q(\zeta_p, \sqrt{2}) : \Q] = 2(p-1)$. Because $[\Q((1+\sqrt{2})^{1/p}) : \Q] = 2p$ by Corollary \ref{corollary2}, we deduce that both $2p$ and $2(p-1)$ divide $[S:\Q]$, so $2p(p-1)$ divides $[S:\Q]$. But since $(1+\sqrt{2})^{1/p}$ is a root of a degree $p$ polynomial over $\Q(\zeta_p, \sqrt{2})$, the degree of $S$ can be no more than $2p(p-1)$, so $[S:\Q] = 2p(p-1)$.
\end{proof}

\begin{theorem} \label{thm:galois group of S}
    We have the following:
    \begin{enumerate}
        \item $\gal(S/\Q(\zeta_p)) \cong D_{2p}$\label{item:galois gp of S over Q zetap}

        \item If $D_{2p} = \langle \sigma, \tau \mid \sigma^p = \tau^2 = 1, \tau^{-1} \sigma \tau = \sigma^{-1} \rangle$, and $\varphi: (\Z/p\Z)^\times \to \aut(D_{2p})$ is the map $k \mapsto \varphi_k$ where $\varphi_k(\tau) = \tau, \varphi_k(\sigma) = \sigma^k$, then $\gal(S/\Q) \cong D_{2p} \rtimes_\varphi (\Z/p\Z)^\times$.
    \end{enumerate}
\end{theorem}
\begin{proof}
Theorem \ref{thm:degree of S} and the tower law imply $[S:\mathbb{Q}(\zeta_p)]=2p$. Note that the minimal polynomial of $(1+\sqrt{2})^{1/p}$ over $\Q(\zeta_p)$ is $(x^p-(1+\sqrt{2}))(x^p-(1-\sqrt{2}))=x^{2p}-x^p-1$, and that $S/\mathbb{Q}(\zeta_p)$ is Galois. Since $S = \mathbb{Q}(\zeta_p)((1+\sqrt{2})^{1/p})$, any automorphism in $\gal(S/\Q(\zeta_p))$ is determined by the root to which it sends $(1+\sqrt{2})^{1/p}$. 

Let $\sigma, \tau \in \gal(S/\Q(\zeta_p))$ be such that
\begin{align*}
\sigma((1+\sqrt{2})^{1/p}) &= \zeta_p (1+\sqrt{2})^{1/p}, \ \ \ \ \ \ \ \ \tau((1+\sqrt{2})^{1/p}) = (1-\sqrt{2})^{1/p}.
\end{align*}
Note that $(\tau \circ \sigma^k)((1+\sqrt{2})^{1/p}) = \tau(\zeta_p^k (1+\sqrt{2})^{1/p}) = \zeta_p^k (1-\sqrt{2})^{1/p}$. Also, $(\sigma^k \circ \tau)((1+\sqrt{2})^{1/p}) = \sigma^k((1-\sqrt{2})^{1/p}) = \sigma^k\left(-(1+\sqrt{2})^{-1/p})\right) = - \zeta_p^{-k} (1+\sqrt{2})^{-1/p} = \zeta_p^{-k} (1-\sqrt{2})^{1/p}$. Hence, $\tau \circ \sigma^k = \sigma^{-k} \circ \tau$. Also, these computations show any element of $\gal(S/\Q(\zeta_p))$ can be written uniquely as either $\sigma^k$ or $\tau \circ \sigma^k$. Therefore, $\gal(S/\Q(\zeta_p))$ is isomorphic to the dihedral group $D_{2p}$, generated by $\sigma$ and $\tau$. This gives \eqref{item:galois gp of S over Q zetap}.

We turn to computing $\gal (S/\Q)$. Each automorphism is determined by where it sends $\zeta_p$ and $(1+\sqrt{2})^{1/p}$. As there are $(p-1)\cdot 2p$ combinations for these possibilities, each combination must be realized by some automorphism. For $g \in \gal(S/\Q(\zeta_p))$, we denote the automorphism that sends $\zeta_p$ to $\zeta_p^k$ and $(1+\sqrt{2})^{1/p}$ to $g((1+\sqrt{2})^{1/p})$ by the tuple $(g,k)$. Note that
\begin{align*}
((g, k) \circ (\sigma^{j'}, k'))\left((1+\sqrt{2})^{1/p}\right) &= (g, k)(\zeta_p^{j'} (1+\sqrt{2})^{1/p})\\
&= \zeta_p^{kj'} g\left((1+\sqrt{2})^{1/p}\right),
\end{align*}
and
\begin{align*}
(g, k) \circ (\tau \sigma^{j'}, k')\left((1+\sqrt{2})^{1/p}\right) &= (g, k) (\zeta_p^{j'}(1-\sqrt{2})^{1/p})\\
&= \zeta_p^{kj'} (g, k)\left(-\frac{1}{(1+\sqrt{2})^{1/p}}\right)\\
&= \zeta_p^{kj'} \frac{-1}{g((1+\sqrt{2})^{1/p})}.
\end{align*}
Thus, we obtain the multiplication table
\begin{align*}
\begin{tabular}{|L|L L|}
\hline
\circ & (\sigma^{j'}, k') & (\tau \sigma^{j'}, k')\\
\hline
(\sigma^j, k) & (\sigma^{kj'+j}, kk') & (\tau \sigma^{kj'-j}, kk')\\
(\tau \sigma^j, k) & (\tau \sigma^{kj'+j}, kk') & (\sigma^{kj'-j}, kk')\\
\hline
\end{tabular}
\end{align*}
This multiplication table agrees with that of $D_{2p} \rtimes_\varphi (\Z/p\Z)^\times$.
\end{proof}

\begin{remark}
One finds the center of the group $G = D_{2p} \rtimes_\varphi (\Z/p\Z)^\times$ is the order $2$ subgroup $Z(G) = \langle (\tau, -1) \rangle$. Also, $N := \{(\sigma^j, k): 0 \leq j < p, 1 \leq k < p\}$ is a normal subgroup of $G$ such that $Z(G) N = G$, so in fact $G$ decomposes as the direct product of $Z(G)$ and $N$. Notice $N \cong \aff(\F_p)$, so another way to view $G$ is as $(\Z/2\Z) \times \aff(\F_p)$.
\end{remark}

Now that we understand the Galois group of $S/\Q$, we can determine $\gal (L/\Q)$.

\begin{theorem}\label{thm:gal gp of L over Q}
We have the following:
    \begin{enumerate}
        \item With the tuple notation from the proof of Theorem \ref{thm:galois group of S},
        \begin{align}
        (\tau^s \sigma^j, k)(\rho_i) = \rho_{(-1)^s(ki-2rj)} \label{eq:galois group on roots},
        \end{align}

        \item $[L:\Q] = p(p-1)$,

        \item $\gal(L/\Q) \cong \aff(\F_p)$.
    \end{enumerate}
\end{theorem}
\begin{proof}
Recall from \eqref{eq505} that
\begin{align*}
\rho_i &= \sqrt{2} \frac{(1-\sqrt{2})^{r/p} \zeta_p^i + (1+\sqrt{2})^{r/p}}{(1-\sqrt{2})^{r/p} \zeta_p^i - (1+\sqrt{2})^{r/p}}.
\end{align*}
Observe that $(\sigma^j,k)(\sqrt{2}) = \sqrt{2}$, and $(\tau \sigma^j,k)(\sqrt{2}) = -\sqrt{2}$. Therefore,
\begin{align*}
(\sigma^j, k)(\rho_i) &= \sqrt{2} \frac{\zeta_p^{-rj} (1-\sqrt{2})^{r/p} \zeta_p^{ki} + \zeta_p^{rj} (1+\sqrt{2})^{r/p}}{\zeta_p^{-rj}(1-\sqrt{2})^{r/p} \zeta_p^{ki} - \zeta_p^{rj} (1+\sqrt{2})^{r/p}}= \rho_{ki-2rj},\\
(\tau \sigma^j, k)(\rho_i) &= -\sqrt{2} \frac{\zeta_p^{-rj} (1+\sqrt{2})^{r/p} \zeta_p^{ki} + \zeta_p^{rj} (1-\sqrt{2})^{r/p}}{\zeta_p^{-rj}(1+\sqrt{2})^{r/p} \zeta_p^{ki} - \zeta_p^{rj} (1-\sqrt{2})^{r/p}}\\
&= \sqrt{2} \frac{\zeta_p^{rj}(1-\sqrt{2})^{r/p} + \zeta_p^{-rj}(1+\sqrt{2})^{r/p} \zeta_p^{ki}}{\zeta_p^{rj}(1-\sqrt{2})^{r/p}-\zeta_p^{-rj}(1+\sqrt{2})^{r/p}\zeta_p^{ki}}= \rho_{-ki+2rj}.
\end{align*}
This proves \eqref{eq:galois group on roots}.

From \eqref{eq:galois group on roots}, we see that each $\rho_i$ is fixed under $(\tau, -1)$. This implies that $L$ is contained in the fixed field of $\langle (\tau, -1) \rangle$, and in particular is a proper subfield of $S$. We have $S = L(\sqrt{2})$ by Corollary \ref{corollary3}, so $\sqrt{2} \not\in L$ and $[S : L] = 2$, hence $L$ is the fixed field of $\langle (\tau, -1) \rangle$. By the tower theorem and Theorem \ref{thm:degree of S}, we then see $[L:\Q] = p(p-1)$. (In fact, one can show that the $\rho_i$ for $i \neq 0$ are Galois conjugates of each other over $\Q(\theta)$, hence their minimal polynomial is $f_{r,p}(x)/(x-\theta)$. Thus, after adjoining any two roots of $f_{r,p}$, we obtain all of them).

From \eqref{eq:galois group on roots}, we see that the permutations of the roots induced by $\gal(L/\Q)$ are precisely the permutations of the form $\rho_i \mapsto \rho_{ki+b}$ for $k \in (\Z/p\Z)^\times, b \in \Z/p\Z$. (Indeed, one sees from \eqref{eq:galois group on roots} that every permutation induced by $\gal(L/\Q)$ is of this form, but there are only $p(p-1)$ such permutations and $|\gal(L/\Q)|=p(p-1)$, so these must be all the permutations.) Thus, $\gal(L/\Q)$ is isomorphic to the group $\aff(\F_p)$ of affine transformations in $\F_p$. (One can also see this via the isomorphism $\gal(L(\sqrt{2})/\Q(\sqrt{2})) \cong \gal(L/\Q)$, obtained by restricting an automorphism of $L(\sqrt{2})$ to $L$, and then noting that $L(\sqrt{2}) = S$ is the splitting field of $x^p-(1+\sqrt{2})$ over $\Q(\sqrt{2})$. One can check such an extension has Galois group $\aff(\F_p)$.)
\end{proof}

\begin{theorem}
    We have $\Q(\theta) = \Q(\alpha)$, where $\alpha := (1+\sqrt{2})^{1/p} + (1-\sqrt{2})^{1/p}$ depends only on $p$ and not on $r$.
\end{theorem}

\begin{proof}
    From \eqref{eq:galois group on roots}, we see $\theta = \theta_{r,p} = \rho_0$ is fixed by $(\tau^s \sigma^j, k)$ precisely when $j$ is zero, so $\gal(S/\Q(\theta))$ is the subgroup $H$ of elements of the form $(\tau^s, k)$. Thus $\Q(\theta)$ is the fixed field of $H$.
    
    Put $\lambda = (1+\sqrt{2})^{1/p}$. Then $\alpha = \lambda - \frac{1}{\lambda}$. The conjugates of $\alpha$ are $\alpha_j := \zeta_p^j \lambda - \frac{\zeta_p^{-j}}{\lambda}$. Note that the imaginary part of $\alpha_j$ is $\sin\left(\frac{2\pi j}{p}\right) \left(\lambda + \frac{1}{\lambda}\right)$, and the real part is $\cos\left(\frac{2\pi j}{p}\right) \left(\lambda - \frac{1}{\lambda}\right)$. We then see that if $i \not \equiv j \pmod{p}$, then $\alpha_i \neq \alpha_j$. Moreover, $(\tau^s \sigma^j, k)(\alpha) = \zeta_p^j (1 + (-1)^s \sqrt{2})^{1/p} + \zeta_p^{-j} (1 - (-1)^s \sqrt{2})^{1/p} = \alpha_{(-1)^s j}$, so $(\tau^s \sigma^j, k)$ fixes $\alpha$ precisely when $j$ is zero. Therefore, $\Q(\alpha)$ is also the fixed field of $H$.
\end{proof}

To determine the ring of integers of $\Q(\theta) = \Q(\alpha)$, we note that $\alpha$ is an algebraic integer, so it is helpful to find the discriminant of the minimal polynomial of $\alpha$ over $\Q$ (which has degree $p$ since $[\Q(\alpha):\Q] = [\Q(\theta):\Q] = p$).

\begin{theorem}\label{thm:disc of alpha}
The discriminant of the minimal polynomial of $\alpha$ is $(-1)^{\frac{p-1}{2}} 2^{\frac{3}{2}(p-1)} p^p$.
\end{theorem}
\begin{proof}
Again, put $\lambda = (1+\sqrt{2})^{1/p}$. Since the roots of the minimal polynomial of $\alpha$ are $\zeta_p^i \lambda - \zeta_p^{-i} \lambda^{-1}$ for $i = 0, \dots, p-1$, the discriminant is
\begin{align*}
&(-1)^{\frac{p(p-1)}{2}} \prod_{i=0}^{p-1} \prod_{\substack{j = 0\\j \neq i}}^{p-1} (\zeta_p^i \lambda - \zeta_p^{-i}\lambda^{-1} - \zeta_p^j \lambda + \zeta_p^{-j}\lambda^{-1})\\
&= (-1)^{\frac{p-1}{2}} \prod_{i=0}^{p-1} \prod_{\substack{j = 0\\j \neq i}}^{p-1} \zeta_p^{-j} (\zeta_p^{2j} (\zeta_p^{i-j}-1)\lambda - (\zeta_p^{j-i}-1)\lambda^{-1}).
\end{align*}
Now the product of $\zeta_p^{-j}$ over all pairs $i, j$ with $i \neq j$ is $\zeta_p^{\frac{p(p-1)}{2} \cdot (p-1)} = 1$, so we can drop that term from the product. Then we reindex the product, letting $k = i-j$ and $\ell = 2j$, so it becomes
\begin{align*}
&(-1)^{\frac{p-1}{2}} \prod_{k=1}^{p-1} \prod_{\ell=0}^{p-1} (\zeta_p^\ell (\zeta_p^k-1)\lambda - (\zeta_p^{-k}-1)\lambda^{-1})\\
&= (-1)^{\frac{p-1}{2}} \prod_{k=1}^{p-1} (\zeta_p^k-1)^p \prod_{\ell=0}^{p-1} (\zeta_p^\ell \lambda + \zeta_p^{-k} \lambda^{-1}).
\end{align*}
Notice the inner product is the polynomial $x^p+\lambda^p = \prod_{\ell=0}^{p-1} (x+\zeta_p^\ell \lambda)$ evaluated at $\zeta_p^{-k} \lambda^{-1}$. Therefore, it equals $\lambda^{-p} + \lambda^p = \sqrt{2}-1 + \sqrt{2}+1 = 2\sqrt{2}$. Moreover, $\prod_{k=1}^{p-1} (\zeta_p^k-1) = \prod_{k=1}^{p-1} (1-\zeta_p^k)$ is the $p$-th cyclotomic polynomial $1 + x + \dots + x^{p-1}$ evaluated at $1$, so it equals $p$. Thus, the product is $(-1)^{\frac{p-1}{2}} (2\sqrt{2})^{p-1} p^p = (-1)^{\frac{p-1}{2}} 2^{\frac{3}{2} (p-1)} p^p$.
\end{proof}

\begin{remark}
Using similar methods, one can show the discriminant of $f_{r,p}(x)$ is\[(-1)^{\frac{p-1}{2}} 2^{\frac{3}{2}(p-1)(p-2)} p^p
\]
Thus $\theta$ generates a subring of the ring of integers of greater index than $\alpha$.
\end{remark}

Thus, we know that the discriminant of $\Q(\theta) = \Q(\alpha)$ divides $(-1)^{\frac{p-1}{2}} 2^{\frac{3}{2}(p-1)} p^p$ by a factor of a square \cite{Simon2001}. From computations, it appears that the discriminant of $\Q(\theta)$ is often the same as the discriminant of $\alpha$, but not always. For instance, when $p = 13, 31$ the power of $p$ in the discriminant is $p^{p-2}$, not $p^p$. However, in all cases we have computed the power of $2$ is $2^{\frac{3}{2}(p-1)}$, so we make the following conjecture:

\begin{conjecture}
The power of $2$ in the discriminant of $\Q(\theta)$ is $2^{\frac{3}{2}(p-1)}$.
\end{conjecture}

When the discriminant of $\Q(\theta)$ is the same as the discriminant of $\alpha$, then the ring of integers of $\Q(\theta)$ is $\Z[\alpha]$. It would be interesting to understand why the discriminant of $\Q(\theta)$ sometimes has a lower power of $p$ than that of $\alpha$, and what the ring of integers is in that case.

\section{Linear forms in logarithms} \label{sec:Linear forms in logarithms}

\subsection{Overview of this section}

The main result of this section is the following.

\begin{theorem}\label{thm:no nontrivial solutions p and y big}
    There are no nontrivial solutions to $x^2-2 = y^p$ for $p > 1951$. If $911 < p \leq 1951$, then any nontrivial solution has $y$ less than the value given in Table \ref{Tab:parameters}.
\end{theorem}

In Section \ref{sec:small y}, we will rule out solutions for the values of $y$ under the lower bound in Table \ref{Tab:parameters}, thus proving Theorem \ref{thm:only triv solns for p big}.

We proceed in subsection \ref{subsec:upper bound} by defining a linear form in two logarithms of algebraic numbers and then derive an upper bound that is exponentially small in its coefficients. The results of Laurent \cite{Laurent2008} can then be used to obtain a lower bound on the linear form that can contradict the upper bound for sufficiently large values of $p$. (Compare also with the arguments in \cite[Section 4.3]{Barros2010}.)

The proof of Theorem \ref{thm:no nontrivial solutions p and y big} proceeds in several stages. First, we obtain the following initial bound on $p$ in subsection \ref{subsec:initial bound}, without making any assumption on $r$ beyond what is stated in Theorem \ref{thm:reduction to thue equation}.

\begin{theorem} \label{thm:initial bound}
    There are no nontrivial solutions to $x^2-2 = y^p$ for $p > 6949$. That is, if $p > 6949$, then every solution to \eqref{eq:main lebesgue nagell eqn} is trivial.
\end{theorem}

In subsection \ref{subsec:r is 1}, we use a method of Bugeaud, Mignotte, and Siksek \cite[Proposition 15.7.1]{Coh2007}, which relies on the modularity of elliptic curves, to prove the following.

\begin{theorem} \label{thm: r is 1}
    Let $x, y, p \in \Z$, $p\geq 17$ a prime such that $x^2-2 = y^p$. Let $r$ be as in Theorem \ref{thm:reduction to thue equation}. If $p < 20000$, then $r = \pm 1$.
\end{theorem}

In the next stage of the proof, we use Theorem \ref{thm: r is 1} to obtain stronger estimates on the linear form.

Throughout this section, we consider the general Thue equation \eqref{eq3}. We may assume $\abs{r} \leq \frac{p-1}{2}$. We will also assume $x > 0$ by possibly negating $x$.

We make frequent use of the following simple lower bound on $y$ without further comment.

\begin{proposition}
    Let $x, y \in \Z$, $p \geq 3$ be a prime such that $x^2-2 = y^p$. If $y \neq -1$, then $y \geq 23$.
\end{proposition}
\begin{proof}
Note that since $x^2-2 = y^p$ and we are assuming $y \neq -1$, we must have $y > 0$. By Theorem \ref{thm:elementary 3} and part \eqref{item:elementary 5.4} of Theorem \ref{thm:elementary 5}, either $y \geq 23$ or $y = 7$ and $p \equiv 1 \pmod{3}$. However, we can easily rule out the case $y = 7$ and $p \equiv 1 \pmod{3}$ using elliptic curves. Indeed, if $p \equiv 1 \pmod{3}, y = 7$ then $y^p = 7z^3$ for $z = 7^{\frac{p-1}{3}}$, so we have $x^2-2 = 7z^3$. Multiplying through by $7^2$ and setting $Y = 7x, X = 7z$ we have $Y^2 = X^3 + 98$. The integral points on this elliptic curve are $(X, Y) = (7, \pm 21)$ \cite{lmfdb:28224.dp2}, which implies $z = 1$ and $ p = 1$, a contradiction.
\end{proof}

\begin{remark}
We could continue ruling out small values of $y$ using this method, but we will eventually get a much larger lower bound on $y$ anyway in Section \ref{sec:small y}, so $y \geq 23$ will suffice for now.
\end{remark}

The key bounds on linear forms in logarithms we use in this work are Theorems 1 and 2 from \cite{Laurent2008}. We give some notation before restating these results. Suppose $\alpha_1, \alpha_2 \in \C$ are nonzero algebraic numbers and $b_1, b_2 \in \N$ are positive integers. Let $D = [\Q(\alpha_1, \alpha_2):\Q]/[\R(\alpha_1, \alpha_2):\R]$, and consider the linear form
\begin{align*}
    \Lambda &= b_2 \log \alpha_2 - b_1 \log \alpha_1.
\end{align*}

\begin{proposition}[{\cite[Theorem 1]{Laurent2008}}]\label{thm:LaurentThm1}
    Let $K, L, R_1, R_2, S_1, S_2 \in \N$, $K \geq 2$. Let $\varrho, \mu \in \R$ with $\varrho > 1, \frac{1}{3} \leq \mu \leq 1$. Put
    \begin{align*}
        &R = R_1 + R_2 - 1, &&S = S_1 + S_2 - 1, &&N = KL, &&g = \frac{1}{4}-\frac{N}{12RS}
    \end{align*}
    \begin{align*}
        &\sigma = 1-\frac{(\mu-1)^2}{2}, &&b = \frac{(R-1)b_2+(S-1)b_1}{2}\left(\prod_{k=1}^{K-1}k!\right)^{-2/(K^2-K)}
    \end{align*}
    \begin{align*}
        \epsilon(N) &= 2\log(N!N^{-N+1}(e^N+(e-1)^N))/N.
    \end{align*}
    Let $a_1, a_2 \in \R_{>0}$ such that
    \begin{align*}
        a_i \geq \varrho \abs{\log \alpha_i} - \log \abs{\alpha_i} + 2Dh(\alpha_i)
    \end{align*}
    for $i=1, 2$. Suppose that
    \begin{align}
        &\#\{\alpha_1^t \alpha_2^s: 0 \leq t < R_1, 0 \leq s < S_1\} \geq L, \label{eq:thm1cnd1}\\
        &\#\{tb_2+sb_1: 0 \leq t < R_2, 0 \leq s < S_2\} > (K-1)L, \label{eq:thm1cnd2}
    \end{align}
    and
    \begin{align}
        K(\sigma L-1) \log \varrho - (D+1)\log N-D(K-1)\log b-gL(Ra_1+Sa_2) > \epsilon(N). \label{eq:thm1cnd3}
    \end{align}
    Then
    \begin{align}
        \abs{\Lambda'} > \varrho^{-\mu K L} \quad \text{ with } \quad \Lambda' = \Lambda \max\left\lbrace\frac{LSe^{LS\abs{\Lambda}/(2b_2)}}{2b_2}, \frac{LRe^{LR\abs{\Lambda}/(2b_1)}}{2b_1}\right\rbrace. \label{eq:thm1 bound}
    \end{align}
    
\end{proposition}

By specializing the values of the parameters in Proposition \ref{thm:LaurentThm1}, one obtains the following weaker result.

\begin{proposition}[{\cite[Theorem 2]{Laurent2008}}]\label{thm:LaurentThm2}
    Suppose $\alpha_1$ and $\alpha_2$ are multiplicatively independent (that is, if $m,n \in \mathbb{Z}$ with $\alpha_1^m \alpha_2^n = 1$, then $m = n = 0$). Let $a_1, a_2, h, \varrho, \mu \in \R$ with $\varrho > 1$ and $\frac{1}{3} \leq \mu \leq 1$. Put
    \begin{align*}
        &\sigma = 1-\frac{(\mu-1)^2}{2}, &&\lambda = \sigma \log \varrho, &&H = \frac{h}{\lambda} + \frac{1}{\sigma},
    \end{align*}
    \begin{align*}
        &\omega = 2\left(1+\sqrt{1+\frac{1}{4H^2}}\right), &&\theta = \sqrt{1+\frac{1}{4H^2}}+\frac{1}{2H},
    \end{align*}
    \begin{align*}
    C &= \frac{\mu}{\lambda^3 \sigma} \left(\frac{\omega}{6} + \frac{1}{2} \sqrt{\frac{\omega^2}{9} + \frac{8\lambda \omega^{5/4} \theta^{1/4}}{3\sqrt{a_1 a_2} H^{1/2}} + \frac{4}{3} \left(\frac{1}{a_1} + \frac{1}{a_2}\right) \frac{\lambda \omega}{H}}\right)^2,\\
    C' &= \sqrt{\frac{C \sigma \omega \theta}{\lambda^3 \mu}}.
    \end{align*}
    Assume that
    \begin{align}
        &h \geq \max\left\lbrace D\left(\log\left(\frac{b_1}{a_2}+\frac{b_2}{a_1}\right)+\log \lambda + 1.75\right)+0.06, \lambda, \frac{D\log 2}{2}\right\rbrace \label{eq:thm2cnd1}\\
        &a_i \geq \max\{1, \varrho\abs{\log \alpha_i}-\log \abs{\alpha_i}+2Dh(\alpha_i)\} \quad (i=1,2),\label{eq:thm2cnd2}\\
        &a_1 a_2 \geq \lambda^2. \label{eq:thm2cnd3}
    \end{align}
    Then
    \begin{align}
        \log \abs{\Lambda} &\geq -C\left(h + \frac{\lambda}{\sigma}\right)^2 a_1 a_2 - \sqrt{\omega \theta} \left(h+\frac{\lambda}{\sigma}\right) - \log\left(C' \left(h + \frac{\lambda}{\sigma}\right)^2 a_1 a_2\right). \label{eq:thm2 bound}
    \end{align}
\end{proposition}

\begin{remark}
    The number $\lambda$ in the statement of the proposition is unrelated to the number $\lambda$ in the proof of Theorem \ref{thm:disc of alpha} above.
\end{remark}

\subsection{The upper bound} \label{subsec:upper bound}

In order to utilize linear forms in logarithms, we need to show that a linear combination of logarithms of algebraic numbers is exponentially small in its coefficients. Recalling Theorem \ref{thm:reduction to thue equation}, the linear form we use is
\begin{align}
    \Lambda &= \log\left(\frac{x+\sqrt{2}}{x-\sqrt{2}}\right) = \log\left(\left(\frac{1+\sqrt{2}}{1-\sqrt{2}}\right)^r \cdot \left(\frac{a+b\sqrt{2}}{a-b\sqrt{2}}\right)^p\right).\label{eq:def of lambda}
\end{align}
Note that $\Lambda > 0$ since we assume $x$ is positive. An upper bound on $\Lambda$ is provided by the following theorem.

\begin{theorem}\label{Thm}
    Let $x^2-2 = y^p$ and $a, b, r \in \Z$ as in Theorem \ref{thm:reduction to thue equation}, with $y \neq -1$. Then with $\Lambda$ as in \eqref{eq:def of lambda},
    \begin{align}
\log \Lambda < 1.053 - (p / 2)\log y. \label{eq8}
\end{align}
\end{theorem}

\begin{proof}
We have
\begin{align}\label{eq7}
e^\Lambda - 1 &= \frac{x+\sqrt{2}}{x-\sqrt{2}}-1 = \frac{2\sqrt{2}}{x-\sqrt{2}}. 
\end{align}
Note $x = \sqrt{y^p+2} > y^{p/2}$, so $x-\sqrt{2} > y^{p/2}-\sqrt{2} = y^{p/2}\left(1-\frac{\sqrt{2}}{y^{p/2}}\right)$. Since $p \geq 3$ and $y \geq 23$, we have $\frac{\sqrt{2}}{y^{p/2}} \leq \frac{\sqrt{2}}{23^{3/2}}$. By plugging this into \eqref{eq7}, we obtain the upper bound
\begin{align}\label{eq35}
\Lambda < e^\Lambda-1 < \frac{2\sqrt{2}}{1-\frac{\sqrt{2}}{23^{3/2}}} y^{-p/2} < 2.866 e^{-(p / 2) \log y}.
\end{align}
Taking logarithms, we obtain \eqref{eq8}.
\end{proof}

Thus, we have an upper bound for $\Lambda$, and the theory of linear forms in logarithms will provide a lower bound.

\subsection{Setup for linear forms in logarithms}\label{subsec:setup for linear forms in logarithms}

First, we prove some facts about the assumed solution $a$ and $b$ of \eqref{eq3}.

\begin{lemma} \label{lemma3.1}
Let $x, y, a, b, r$ be as in \eqref{eq5}, $y \neq -1$. Then $y = (-1)^r (a^2-2b^2)$, and $a,b$ are nonzero and coprime.
\end{lemma}

\begin{proof}
Note that $y^p = (x+\sqrt{2})(x-\sqrt{2})$ is the norm from $\Z[\sqrt{2}]$ to $\Z$ of $x+\sqrt{2}$, so taking the norm of \eqref{eq5}, we find $y^p = (-1)^r(a^2-2b^2)^p$. Since both $y$ and $(-1)^r(a^2-2b^2)$ are real, $y = (-1)^r(a^2-2b^2)$.

Note that equation \eqref{eq5} implies that any common divisor of $a$ and $b$ divides $1$, so $a$ and $b$ are coprime. If $a = 0$, then $y = (-1)^{r+1} 2b^2$, contradicting that $y$ is odd by Theorem \ref{thm:elementary 1}. If $b = 0$, then since $\gcd(a, b) = 1$, $a = \pm 1$. Then $y = (-1)^r$, so $x^2-2 = (-1)^r$ which implies $y = -1, x = \pm 1$, and these are the trivial solutions.
\end{proof}

\begin{lemma}\label{lemma2}
Let $x, y, a, b, r$ be as in \eqref{eq5}, and assume $x > 0, y \neq 1$.
\begin{enumerate}
	
	\item $\sgn(a+b\sqrt{2}) = 1, \sgn(a-b\sqrt{2}) = (-1)^r$.
	
	\item $r \neq 0$
	
	\item If $r < 0$, then $\sgn(a) = \sgn(b) = 1$.

	\item If $r > 0$, then $\sgn(a) = (-1)^r, \sgn(b) = (-1)^{r+1}$.

\end{enumerate}
\end{lemma}

\begin{proof}
First, since $y \geq 23$, $p \geq 3$, we have $x > 100$. We have $x+\sqrt{2} = (1+\sqrt{2})^r (a+b\sqrt{2})^p$ by \eqref{eq5}, so $a+b\sqrt{2} > 0$. Also, $x-\sqrt{2} = (1-\sqrt{2})^r (a-b\sqrt{2})^p$. Since $1-\sqrt{2}$ is negative, we must have $\sgn(a-b\sqrt{2}) = (-1)^r$. If $r$ is even, then $a = \frac{1}{2}(a+b\sqrt{2}+a-b\sqrt{2}) > 0$. If $r$ is odd, then $b = \frac{1}{2\sqrt{2}}(a+b\sqrt{2}-(a-b\sqrt{2})) > 0$.

Now, if $r \leq 0$, then since $x+\sqrt{2} > x-\sqrt{2}$ and $(\sqrt{2}-1)^r \geq (\sqrt{2}+1)^r$, we have
\begin{align}
\abs{a-b\sqrt{2}}^p = \frac{x-\sqrt{2}}{(\sqrt{2}-1)^r} < \frac{x+\sqrt{2}}{(\sqrt{2}+1)^r} = \abs{a+b\sqrt{2}}^p, \label{eq34}
\end{align}
so $\abs{a-b\sqrt{2}} < \abs{a+b\sqrt{2}}$. Therefore, $a$ and $b$ have the same sign, so since either $a > 0$ or $b > 0$, both $a$ and $b$ are positive.

If $r = 0$, then
\begin{align*}
x+\sqrt{2} &= (a+b\sqrt{2})^p > a^p+p a^{p-1}b\sqrt{2} > a^p+2\sqrt{2}\\
&> (a-b\sqrt{2})^p+2\sqrt{2} = x-\sqrt{2}+2\sqrt{2} = x+\sqrt{2}
\end{align*}
a contradiction. Therefore, $r \neq 0$.

Now suppose $r > 0$. Note that the function $\frac{t+\sqrt{2}}{t-\sqrt{2}} = 1+\frac{2\sqrt{2}}{t-\sqrt{2}}$ is decreasing and positive for $t > \sqrt{2}$, so since $x > 2$, we see
\begin{align*}
&\frac{x+\sqrt{2}}{x-\sqrt{2}} < \frac{2+\sqrt{2}}{2-\sqrt{2}} = \frac{\sqrt{2}+1}{\sqrt{2}-1} \leq \left(\frac{\sqrt{2}+1}{\sqrt{2}-1}\right)^r.
\end{align*}
Hence
\begin{align}
    \frac{x-\sqrt{2}}{(\sqrt{2}-1)^r} > \frac{x+\sqrt{2}}{(\sqrt{2}+1)^r},
\end{align}
so by the same reasoning as in \eqref{eq34}, $\abs{a-b\sqrt{2}} > \abs{a+b\sqrt{2}}$, so $a$ and $b$ have opposite signs. If $r$ is even, then $a > 0$ so $b < 0$, and if $r$ is odd, then $b > 0$ so $a < 0$.
\end{proof}

In view of this lemma, it is convenient to introduce the following notation:
\begin{align*}
\epsilon &= -\sgn(r), \ \ \ \ \ \ \epsilon' = (-1)^r.
\end{align*}
Then from Lemma \ref{lemma2}, it follows that
\begin{align}
\Lambda &= 2r \log\left(\sqrt{2}+1\right) + \epsilon p \log\left(\epsilon' \frac{a+\epsilon b\sqrt{2}}{a-\epsilon b \sqrt{2}}\right),\label{eq9}
\end{align}
and the logarithms are all taken of positive real numbers greater than $1$ (here we follow \cite{BPV2024} in arranging the linear form). Note that if $r > 0$, then the first term is positive and the second is negative, while the opposite happens if $r < 0$. Therefore, to match with the notation from \cite{Laurent2008}, we multiply \eqref{eq9} by $-\epsilon$, so that if we define
\begin{align}
\alpha_1 &= \epsilon' \frac{a+\epsilon b \sqrt{2}}{a - \epsilon b \sqrt{2}} = \frac{(a+\epsilon b\sqrt{2})^2}{y},\label{eq:linear form variables}\\
\alpha_2 &= \sqrt{2}+1,\nonumber\\
b_1 &= p, \ \ \ \ \ \ \  b_2 = 2\abs{r},\nonumber
\end{align}
then the linear form becomes $-\epsilon\Lambda = b_2 \log \alpha_2 - b_1 \log \alpha_1$, where $b_1, b_2$ are positive integers and $\alpha_1, \alpha_2$ are algebraic numbers greater than $1$.

\begin{lemma} \label{lemma3.3}
With notation as above,
\begin{align*}
h(\alpha_1) &= \log \abs{a+\epsilon b \sqrt{2}}, \ \ \ \ \ h(\alpha_2) = \frac{1}{2}\log(1+\sqrt{2}).
\end{align*}
\end{lemma}
\begin{proof}
First we compute the minimal polynomial for $\alpha_1$ over $\Z$. Note that $\alpha_1$ is irrational since $ab \neq 0$, so its minimal polynomial over $\Q$ is
\begin{align*}
\left(x-\epsilon'\frac{a+\epsilon b \sqrt{2}}{a - \epsilon b \sqrt{2}}\right) \left(x-\epsilon'\frac{a-\epsilon b \sqrt{2}}{a+\epsilon b \sqrt{2}}\right) &= x^2 - 2\epsilon' \frac{a^2+2b^2}{a^2-2b^2} x + 1,
\end{align*}
and therefore it is also a root of
\begin{align}\label{eq:min poly of alpha1 over Z}
(a^2-2b^2) x^2 - 2\epsilon' (a^2+2b^2) x + (a^2-2b^2).
\end{align}
To show this is the minimal polynomial over $\Z$, it suffices to show the coefficients are coprime. Suppose a prime $q$ divides $a^2-2b^2$ and $2(a^2+2b^2)$. Note $a^2-2b^2 = (-1)^r y$ by Lemma \ref{lemma3.1}, so it is odd by Theorem \ref{thm:elementary 1}. Hence, $q \neq 2$. Thus, $q$ divides $a^2+2b^2$, and since it also divides $a^2-2b^2$, $q$ divides their sum $2a^2$ and their difference $4b^2$, so $q$ divides $a$ and $b$. However, this contradicts the fact that $\gcd(a, b) = 1$ by Lemma \ref{lemma3.1}. Therefore, these coefficients of the polynomial in \eqref{eq:min poly of alpha1 over Z} are coprime, so this is the minimal polynomial of $\alpha_1$ over $\Z$. Therefore,
\begin{align*}
h(\alpha_1) &= \frac{1}{2} \left(\log \abs{a^2-2b^2} + \log \abs{\frac{a+\epsilon b \sqrt{2}}{a-\epsilon b \sqrt{2}}}\right) = \log \abs{a+\epsilon b\sqrt{2}}.
\end{align*}

Next, note that the minimal polynomial for $\alpha_2 = 1+\sqrt{2}$ is $(x-1)^2-2 = x^2-2x-1$. Thus,
\begin{align*}
h(\alpha_2) &= \frac{1}{2} \log(1+\sqrt{2}). \qedhere
\end{align*}
\end{proof}

\subsection{Initial bound}\label{subsec:initial bound}

We shall get our initial bound on $p$ through Proposition \ref{thm:LaurentThm2} since it is significantly simpler than Proposition \ref{thm:LaurentThm1}. To apply this proposition, we must verify the condition that $\alpha_1$ and $\alpha_2$ are multiplicatively independent.

\begin{lemma} \label{lemma3.4}
Let the notation be as above, and assume $y \neq -1$ (that is, assume $y$ comes from a nontrivial solution to \eqref{eq:main lebesgue nagell eqn}). Then $\alpha_1$ and $\alpha_2$ are multiplicatively independent.
\end{lemma}
\begin{proof}
Let $m, n \in \Z$ such that $\alpha_1^m \alpha_2^n = 1$. We may assume without loss of generality that $m \geq 0$. We have $\alpha_1^m \alpha_2^n = 1$ if and only if
\begin{align*}
(\sqrt{2}+1)^n (a+\epsilon b\sqrt{2})^m = ({\epsilon'})^m (a-\epsilon b\sqrt{2})^m.
\end{align*}
Now, $a+b\sqrt{2}$ is not a unit since its norm is $a^2-2b^2$, which is equal to $(-1)^r y$ by Lemma \ref{lemma3.1}, and this is only $\pm 1$ when $y = -1$. Therefore, assuming $m \neq 0$, some prime $\pi$ of $\Z[\sqrt{2}]$ divides both sides. Since $\sqrt{2}+1$ and $\epsilon'$ are units, we must have $\pi \mid (a+b\sqrt{2})$ and $\pi \mid (a-b\sqrt{2})$. Then $\pi \mid 2a$ and $\pi \mid 2\sqrt{2} b$. If $\pi \neq \sqrt{2}$, then $\pi \mid a$ and $\pi \mid b$, so taking norms to $\Z$, if $\norm{\pi} = p^f$ for $p$ a prime of $\Z$, then $p^f \mid a^2$ and $p^f \mid b^2$, so $p \mid a$ and $p \mid b$, contradicting Lemma \ref{lemma3.1}. Therefore, $\pi = \sqrt{2}$. Since $\sqrt{2} \mid (a+b\sqrt{2})$, taking norms again we have $2 \mid y$, contradicting Theorem \ref{thm:elementary 1}. Therefore, $m = 0$, so $n = 0$ as well.
\end{proof}

Next, to apply Proposition \ref{thm:LaurentThm2}, we must choose the five parameters $a_1, a_2, h, \varrho, \mu \in \R$ with $\varrho > 1$ and $\frac{1}{3} \leq \mu \leq 1$. We choose $a_1, a_2,h$ in terms of $\varrho, \mu, y$.

\begin{lemma} \label{thm:Thm2 choice of a1 a2 h}
We may choose
\begin{align}
a_1 &= 0.9(\varrho+1) + 2\log y, \ \ \ \ a_2 = (\varrho+1) \log(\sqrt{2}+1), \label{eq12}
\end{align}
to satisfy \eqref{eq:thm2cnd2} and
\begin{align}
h &= 2 \left(\log\left(\frac{p}{a_2} + \frac{p}{a_1}\right) + \log \lambda + 1.78\right) \label{eq13}
\end{align}
to satisfy \eqref{eq:thm2cnd1}, under the assumption that, with this value of $h$,
\begin{align}
    h \geq \max(\lambda, \log 2). \label{eq:h max assumption}
\end{align}
\end{lemma}
\begin{proof}
Note that the value of $D$ is $[\Q[\alpha_1, \alpha_2]:\Q]/[\R(\alpha_1,\alpha_2):\R] = [\Q[\sqrt{2}]:\Q]/[\R:\R] = 2$. Substituting our values of $\alpha_i$ and $h(\alpha_i)$ from Lemma \ref{lemma3.3} into \eqref{eq:thm2cnd2}, we must satisfy the bounds
\begin{align*}
a_1 &\geq \max\left\lbrace 1, (\varrho-1) \log \abs{\frac{a+\epsilon b \sqrt{2}}{a-\epsilon b \sqrt{2}}} + 4 \log \abs{a+\epsilon b \sqrt{2}}\right\rbrace,\\
a_2 &\geq \max\left\lbrace 1, (\varrho-1) \log(\sqrt{2}+1) + 2 \log(\sqrt{2}+1)\right\rbrace.
\end{align*}
Notice the bound on $a_2$ simplifies to
\begin{align*}
a_2 &\geq \max\{1, (\varrho+1) \log(\sqrt{2}+1)\}.
\end{align*}
Also, by rearranging \eqref{eq9},
\begin{align}
\log \abs{\frac{a+\epsilon b\sqrt{2}}{a-\epsilon b\sqrt{2}}} &= \frac{2\abs{r} \log(\sqrt{2}+1) + \epsilon \Lambda}{p}. \label{eq11}
\end{align}
By \eqref{eq8}, using $y \geq 23$ and $p \geq 3$, we get $\frac{\Lambda}{p} < 0.009$. Since we have chosen $r$ so that $\abs{r} < \frac{p}{2}$,
\begin{align*}
\log \abs{\frac{a+\epsilon b\sqrt{2}}{a-\epsilon b\sqrt{2}}} < \log(\sqrt{2}+1) + 0.009 < 0.9.
\end{align*}

Now, by Lemma \ref{lemma3.1}, we have
\begin{align*}
&\log \abs{a+\epsilon b\sqrt{2}} + \log \abs{a-\epsilon b \sqrt{2}} = \log \abs{a^2-2b^2} = \log y,
\end{align*}
and therefore
\begin{align*}
    \log \abs{a+\epsilon b \sqrt{2}} = \log y - \log \abs{a-\epsilon b \sqrt{2}}.
\end{align*}
It follows that
\begin{align*}
&(\varrho-1) \log \abs{\frac{a+\epsilon b \sqrt{2}}{a-\epsilon b \sqrt{2}}} + 4 \log \abs{a+\epsilon b \sqrt{2}}\\
&= (\varrho-1) \log \abs{\frac{a+\epsilon b \sqrt{2}}{a-\epsilon b \sqrt{2}}} + 2\log \abs{a+\epsilon b\sqrt{2}} + 2\log y - 2\log \abs{a-\epsilon b\sqrt{2}}\\
&= (\varrho+1) \log \abs{\frac{a+\epsilon b\sqrt{2}}{a-\epsilon b\sqrt{2}}} + 2\log y\\
&< 0.9 (\varrho+1) + 2\log y.
\end{align*}
Also, the given choices of $a_1$ and $a_2$ are greater than $1$ because $\varrho > 1$. Thus, they satisfy \eqref{eq:thm2cnd2}.

In order to satisfy \eqref{eq:thm2cnd1}, we must choose $h$ so that
\begin{align*}
h &\geq \max\left\lbrace D \left(\log\left(\frac{b_1}{a_2} + \frac{b_2}{a_1}\right) + \log \lambda + 1.75\right) + 0.06, \lambda, \frac{D \log 2}{2}\right\rbrace\\
&= \max\left\lbrace 2 \left(\log\left(\frac{p}{a_2} + \frac{2\abs{r}}{a_1}\right) + \log \lambda + 1.75\right) + 0.06, \lambda, \log 2\right\rbrace.
\end{align*}
Thus, assuming the first expression is the maximum, the given value of $h$ satisfies \eqref{eq:thm2cnd1} since $\abs{r} < \frac{p}{2}$.
\end{proof}

We now have $a_1, a_2$, and $h$ expressed as functions of $y, p, \mu, \varrho$, such that the conditions of Proposition \ref{thm:LaurentThm2} are satisfied, assuming $h \geq \max(\lambda, \log 2)$ and $a_1 a_2 \geq \lambda^2$. Proposition \ref{thm:LaurentThm2} then gives a bound of the form $\log \abs{\Lambda} \geq f(y, p, \mu, \varrho)$. Combining the lower bound with the upper bound \eqref{eq8}, if we define
\begin{align*}
g(y, p, \mu, \varrho) &:= 1.053-(p/2) \log y - f(y, p, \mu, \varrho),
\end{align*}
then we must have $g(y, p, \mu, \varrho) > 0$ if $x^2-2 = y^p$ is a nontrivial solution to \eqref{eq:main lebesgue nagell eqn}. (This function $g$ should not be confused with the number $g$ in Proposition \ref{thm:LaurentThm1}.)

Note that as $p$ increases, the main term $-(p/2)\log y$ in $g$ dominates, allowing us to reach a contradiction for sufficiently large $p$. In order to prove this, we make some estimates on the terms in $g$ to simplify it.

\begin{lemma} \label{thm:simplified g bound}
Fix $\mu, \varrho, p^{(\ell)} \in \R$ with $\frac{1}{3} \leq \mu \leq 1, \varrho \geq 1, p^{(\ell)} > 0$. Set $a_2, \sigma, \lambda$ as in Lemma \ref{thm:Thm2 choice of a1 a2 h} and Proposition \ref{thm:LaurentThm2}. Set
\begin{align*}
&a_1^{(\ell)} = 0.9(\varrho+1) + 2\log(23),\\
&h^{(\ell)} = 2\left(\log p^{(\ell)}-\log a_2 + \log \lambda + 1.78\right),\\
&H^{(\ell)} = \frac{h^{(\ell)}}{\lambda} + \frac{1}{\sigma},\\
&\omega^{(u)} = 2\left(1+\sqrt{1+\frac{1}{4 {H^{(\ell)}}^2}}\right),\\
&\theta^{(u)} = \sqrt{1+\frac{1}{4 {H^{(\ell)}}^2}} + \frac{1}{2 H^{(\ell)}},\\
&C^{(u)} = \frac{\mu}{\lambda^3 \sigma} \left(\frac{\omega^{(u)}}{6} + \frac{1}{2} \sqrt{\frac{{\omega^{(u)}}^2}{9} + \frac{8 \lambda {\omega^{(u)}}^{5/4} {\theta^{(u)}}^{1/4}}{3\sqrt{a_2 a_1^{(\ell)}} {H^{(\ell)}}^{1/2}} + \frac{4}{3}\left(\frac{1}{a_2} + \frac{1}{a_1^{(\ell)}}\right) \frac{\lambda \omega^{(u)}}{H^{(\ell)}}}\right)^2,\\
&{C'}^{(u)} = \sqrt{\frac{C^{(u)} \sigma \omega^{(u)} \theta^{(u)}}{\lambda^3 \mu}}.
\end{align*}
Then for $y, p \in \R_{>0}$ with $p \geq p^{(\ell)}$, we have $g(y, p, \mu, \varrho) \leq g^{(u)}(y, p)$, where
\begin{align}
g^{(u)}(y, p) &= C_1 - (p / 2) \log y + C_2 (\log p + C_3)^2 (\log y + C_4) + C_5 \log p \label{eq:g upper bound}\\
&+ \log\left((\log p + C_3)^2 (\log y + C_4)\right), \nonumber
\end{align}
where $C_1$ through $C_5$ are constants (depending on $\mu, \varrho, p^{(\ell)}$) given as follows:
\begin{align*}
C_3 &= \log\left(\frac{1}{a_2} + \frac{1}{a_1^{(\ell)}}\right) + \log \lambda + 1.78 + \frac{\lambda}{2\sigma},\\
C_1 &= 1.053 + 2 \sqrt{\omega^{(u)} \theta^{(u)}} C_3 + \log\left(8 {C'}^{(u)} a_2\right),\\
C_2 &= 8 C^{(u)} a_2, \qquad C_4 = 0.45(\varrho+1), \qquad C_5 = 2 \sqrt{\omega^{(u)} \theta^{(u)}}.
\end{align*}
\end{lemma}

\begin{proof}
It is seen that each of the variables with an $(\ell)$ superscript is a lower bound for the variable of which it is a superscript when $p \geq p^{(\ell)}$; similarly, those with a $(u)$ superscript are an upper bound. Then if we set $a_1, a_2, h$ as in Lemma \ref{thm:Thm2 choice of a1 a2 h}, we have $h \leq h^{(u)}$, where
\begin{align*}
    h^{(u)} = 2\left(\log p+\log\left(\frac{1}{a_2} + \frac{1}{a_1^{(\ell)}}\right) + \log \lambda + 1.78\right).
\end{align*}
Note that $h^{(u)}$ depends on $p$ and $a_1$ depends on $y$.

From the definition of $f$ from \eqref{eq:thm2 bound}, $f(y, p, \mu, \varrho) \geq f^{(\ell)}(y, p)$ where
\begin{align*}
&f^{(\ell)}(y, p)\\
&= -C^{(u)} \left(h^{(u)} + \frac{\lambda}{\sigma}\right)^2 a_1 a_2 - \sqrt{\omega^{(u)} \theta^{(u)}} \left(h^{(u)} + \frac{\lambda}{\sigma}\right) - \log\left({C'}^{(u)} \left(h^{(u)} + \frac{\lambda}{\sigma}\right)^2 a_1 a_2\right),
\end{align*}
and thus $g(y, p, \mu, \varrho) \leq g^{(u)}(y, p)$ where $g^{(u)}(y, p) = 1.053-(p / 2) \log y - f^{(\ell)}(y, p)$, which simplifies to \eqref{eq:g upper bound}.
\end{proof}

\begin{lemma} \label{lemma3.5}
Let $u$ be a function of the form
\begin{align*}
u(y, p) &= C_1 - (p / 2) \log y + C_2 (\log p + C_3)^2 (\log y + C_4) + C_5 \log p\\
&+ \log\left((\log p + C_3)^2 (\log y + C_4)\right),
\end{align*}
where $C_1$ through $C_5$ are positive quantities that do not depend on $y$ or $p$. Assume that $y_0 > e^e, p_0 > e^2$, and \[
(\log p_0 + C_3) \log(\log p_0 + C_3) > 1.
\]
If $u(y_0, p_0) < 0$, then $u(y, p) < 0$ for all $y \geq y_0$ and $p \geq p_0$.
\end{lemma}

\begin{proof}

The main term is $-(p / 2) \log y$. We must show that this dominates for sufficiently large $y$ and $p$. If we divide $u(y, p)$ by $p \log y$, we get
\begin{align*}
\frac{u(y,p)}{p\log y} &=-\frac{1}{2} + \frac{C_1}{p \log y} + \frac{C_2 (\log p + C_3)^2 \left(1+\frac{C_4}{\log y}\right)}{p} + \frac{C_5 \log p}{p \log y}\\
&+ \frac{2\log(\log p + C_3)}{p \log y} + \frac{\log(\log y + C_4)}{p \log y}.
\end{align*}
Every term except the last is always decreasing in $y$. The derivative of the last term with respect to $\log y$ is
\begin{align*}
\frac{1}{p \log y (\log y + C_4)} - \frac{\log(\log y + C_4)}{p (\log y)^2} < \frac{1-\log \log y}{p (\log y)^2},
\end{align*}
so it is decreasing when $y > e^e$. Thus, $u(y, p)/(p \log y)$ is decreasing in $y$ for $y \geq y_0$. The derivative of the third term with respect to $p$ is
\begin{align*}
C_2 \left(1+\frac{C_4}{\log y}\right) \frac{(\log p + C_3)(2-\log p - C_3)}{p^2},
\end{align*}
which is negative when $p > e^2$. The derivative of the fourth term with respect to $p$ is $\frac{C_5(1-\log p)}{p^2 \log y}$ which is negative when $p > e$. The derivative of the fifth term with respect to $p$ is
\begin{align*}
\frac{2-2(\log p + C_3)\log(\log p + C_3)}{(\log p + C_3) p^2 \log y},
\end{align*}
which by assumption is negative. Thus, $u(y, p)/(p \log y)$ is decreasing in both $y$ and $p$.
\end{proof}

\begin{proof}[Proof of Theorem \ref{thm:initial bound}]
See the Mathematica notebook \verb|theorem_5_2.nb| of \cite{Katz_Lebesgue-Nagell_Helper_Code} for the computations used in this proof.

We choose the parameters $\mu$ and $\varrho$ as follows. We fix $y$ at our current lower bound $23$, and assume $p$ is set to the implicit function in $\mu$ and $\varrho$ so that $g(23, p(\mu, \varrho), \mu, \varrho) = 0$ (assuming the upper bound dominates and causes $g$ to decrease in $p$). We then wish to minimize the function $p(\mu, \varrho)$. Then we must have $0 = \frac{\partial p}{\partial \mu} = \frac{\partial p}{\partial \varrho}$, and because $g$ is constant with $p$ set to $p(\mu, \varrho)$, differentiating with respect to $\mu$ yields $\frac{\partial g}{\partial p} \frac{\partial p}{\partial \mu} + \frac{\partial g}{\partial \mu} = 0$. But since we are assuming $\frac{\partial p}{\partial \mu} = 0$, we must have $\frac{\partial g}{\partial \mu} = 0$. Similarly, $\frac{\partial g}{\partial \varrho} = 0$. Therefore, we search for solutions to the following system of $3$ equations in $3$ variables:
\begin{align*}
&\left(g(23, p, \mu, \varrho), \frac{\partial g}{\partial \mu}(23, p, \mu, \varrho), \frac{\partial g}{\partial \varrho}(23, p, \mu, \varrho)\right) = (0, 0, 0)
\end{align*}
and we find the solution $(p, \mu, \varrho) \approx (6950.6, 0.508613, 7.99202)$. Thus, we shall set $\mu = 0.508613, \varrho = 7.99202$. We may assume that $p$ is at least the next prime, so set $p^{(\ell)} = 6959$.

With these choices of parameters, we have $h^{(\ell)} > 18$, which easily exceeds $\lambda$ and $\log 2$, and $a_1^{(\ell)} a_2 > 113$, which easily exceeds $\lambda^2$. Therefore, the assumptions \eqref{eq:thm2cnd3} and \eqref{eq:h max assumption} are satisfied. It is seen that if we evaluate $g^{(u)}$ when $(y_0, p_0) = (23, 6993)$ or $(y_0, p_0) = (31, 6959)$, then it is negative. Then the assumptions of Lemma \ref{lemma3.5} are satisfied, so when $p \geq 6993$ or $y \geq 31$ then $g^{(u)}$ is negative. Therefore, after checking that $g(y, p, \mu, \varrho)$ is negative for the finitely many cases $y = 23, 6959 \leq p < 6993$, we conclude $g(y, p, \mu, \varrho) < 0$ whenever $y \geq 23, p \geq 6959$. Therefore, there are no nontrivial solutions when $p \geq 6959$, so we have the bound $p \leq 6949$.
\end{proof}

\subsection{Computation to prove Theorem \ref{thm: r is 1}} \label{subsec:r is 1}

Here we use modularity techniques. Let $E$ and $F$ be two elliptic curves over $\Q$ with conductors $N$ and $N'$, respectively. If $E$ and $F$ are related via level-lowering (see \cite[Definition 15.2.1]{Coh2007}), then for all prime numbers $\ell$ we have the following:
\begin{enumerate}
\item If $\ell \nmid NN'$, then $a_\ell(E) \equiv a_\ell(F) \pmod{p}$.
\item If $\ell \| N$ and $\ell \nmid N'$, then $a_\ell(F) \equiv \pm (\ell+1) \pmod{p}$.
\end{enumerate}
(See \cite[Proposition 15.2.3]{Coh2007}.)

\begin{proof}[Proof of Theorem \ref{thm: r is 1}]

In \cite[Proposition 15.7.1]{Coh2007}, Siksek described joint work with Bugeaud and Mignotte showing how, using modularity, one can prove for any given prime $p$ that the $r$ of \eqref{eq3} must equal $\pm 1$ by finding a set of auxiliary primes satisfying certain conditions for each of four elliptic curves. This computation proceeds as follows.

By Theorem \ref{thm:thue equations small p}, we may assume $p\geq 17$. We may associate to any solution of \eqref{eq:main lebesgue nagell eqn} the elliptic curve (a ``Frey curve'')
\begin{align}\label{eq:frey curve for lebesgue-nagell eqn}
    Y^2 = X(X^2+2xX+2),
\end{align}
which has minimal discriminant $\Delta_{\text{min}}=2^8y^p$ and conductor $N=2^7\text{rad}(y)$ (see \cite[p. 518]{Coh2007}, which draws upon \cite[Lemma 2.1]{BS2004}). The curve \eqref{eq:frey curve for lebesgue-nagell eqn} is related via level-lowering to a newform of level 128 (see Lemmas 3.2 and Lemma 3.3 of \cite{BS2004}). There are four newforms on $\Gamma_0(128)$ without character, and these correspond via modularity to the elliptic curves with Cremona labels $128A1, 128B1, 128C1, 128D1$. Let $F$ denote one of these elliptic curves. Given $p$ and $F$, we search for a prime $\ell$ such that:
\begin{enumerate}
    \item $\ell = np+1$ for some positive integer $n$,

    \item $\ell \equiv \pm 1 \pmod{8}$,\label{item:make sure ell ha sq rt of 2}

    \item $a_\ell(F) \not\equiv \pm (\ell+1) \pmod{p}$, \label{item:ensure ell not divide y}

    \item $(1+\theta)^n \not\equiv 1 \pmod{\ell}$, where $\theta$ is a square root of $2$ in $\F_\ell$ (this exists by \eqref{item:make sure ell ha sq rt of 2}). \label{item:ensure phi 1 + theta not zero mod p}
\end{enumerate}
Condition \eqref{item:ensure ell not divide y} and modularity ensure $\ell \nmid y$, so that $y^p$ is congruent modulo $\ell$ to an element of
\begin{align*}
    &\mu_n(\F_\ell) = \{\delta \in \F_\ell : \delta^n = 1\},
\end{align*}
and therefore $x$ is congruent modulo $\ell$ to an element of
\begin{align*}
    X_\ell' = \{\delta \in \F_\ell : \delta^2-2 \in \mu_n(\F_\ell)\}.
\end{align*}
For $\delta \in X_\ell'$, let $E_\delta$ be the elliptic curve $Y^2 = X^3 + 2\delta X^2 + 2X$ over $\F_\ell$, and note that, by the modularity result at the beginning of the subsection, we have that $x$ is congruent modulo $\ell$ to an element of
\begin{align*}
    X_\ell = \{\delta \in X_\ell' : a_\ell(E_\delta) \equiv a_\ell(F) \pmod{p}\}.
\end{align*}
From \eqref{eq5}, we see that if $x \equiv \delta \pmod{\ell}$ with $\delta \in X_\ell$, then
\begin{align*}
    \delta + \theta \equiv (1+\theta)^r (a+b\theta)^p \pmod{\ell},
\end{align*}
and we observe $a+b\theta \not \equiv 0 \pmod{\ell}$. If we let $\Phi$ denote the composition of the discrete logarithm $\F_\ell^\times \to \Z/(\ell-1)\Z$ (with respect to an arbitrary, but fixed, primitive root of $\ell$) with the reduction map $\Z/(\ell-1)\Z \to \Z/p\Z$, then we see $r$ is congruent modulo $p$ to an element of
\begin{align*}
    R_\ell(F) &= \left\lbrace \frac{\Phi(\delta+\theta)}{\Phi(1+\theta)}: \delta \in X_\ell\right\rbrace.
\end{align*}
Here we are using \eqref{item:ensure phi 1 + theta not zero mod p} to ensure $\Phi(1+\theta) \not \equiv 0 \pmod{p}$. Since $|r| < \frac{p}{2}$, in order to show that $r = \pm 1$, it suffices to show that $r \equiv \pm 1 \pmod{p}$.

We find, for each choice of $p$ and $F$, a set of primes $\ell_1, \dots, \ell_k$ satisfying the four conditions above such that
\begin{align*}
    \bigcap_{1 \leq j \leq k} R_{\ell_j}(F) \subseteq \{1, -1\}.
\end{align*}
We verified this computation in Sage and output a text file with the auxiliary primes used in the proof for each prime $11 \leq p < 20000$. The script is contained in the file \verb|theorem_5_3.sage|, and the output in the file \verb|theorem_5_3_primes.txt| of \cite{Katz_Lebesgue-Nagell_Helper_Code}. The computation takes less than thirty minutes.

It does not seem possible to find suitable auxiliary primes $\ell_i$ to show $r = \pm 1$ when $p=7$.
\end{proof}

\subsection{Improved bound with Proposition \ref{thm:LaurentThm2} using Theorem \ref{thm: r is 1}} \label{subsec:improved bound thm 2}

From Theorems \ref{thm:initial bound} and \ref{thm: r is 1}, we may now assume $r = \pm 1$. We will use this information to refine some of the estimates from subsection \ref{subsec:initial bound} to get the following improvement to Theorem \ref{thm:initial bound}.

\begin{theorem}\label{thm:improved bound thm 2}
    There are no nontrivial solutions to $x^2-2 = y^p$ for $p > 1951$.
\end{theorem}

We begin by refining our estimate of \eqref{eq11}, which gives
\begin{align*}
\log \abs{\frac{a+\epsilon b\sqrt{2}}{a-\epsilon b\sqrt{2}}} &= \frac{2 \log(\sqrt{2}+1) + \epsilon \Lambda}{p}.
\end{align*}
By \eqref{eq8}, assuming $p \geq 100$ (which is much smaller than the bound we will obtain), we already get $\Lambda < 10^{-50}$, say, so by the same reasoning as in the proof of Lemma \ref{thm:Thm2 choice of a1 a2 h}, we may set
\begin{align}
a_1 &= (\varrho+1) \frac{2\log(\sqrt{2}+1)+10^{-50}}{p} + 2\log y. \label{eq14}
\end{align}
We also sharpen our choice of $h$ to be
\begin{align*}
h &= 2\left(\log\left(\frac{p}{a_2} + \frac{2}{a_1}\right) + \log \lambda + 1.78\right),
\end{align*}
again under the assumption that $h \geq \max(\lambda, \log 2)$ and $a_1 a_2 \geq \lambda^2$. With these choices, we set $f(y, p, \mu, \varrho)$ and $g(y, p, \mu, \varrho)$ in the same way as in subsection \ref{subsec:initial bound}.

Then given $\mu, \varrho, p^{(\ell)} \in \R, \frac{1}{3} \leq \mu \leq 1, \varrho \geq 1, p^{(\ell)} > 0$, we set
\begin{align*}
a_1^{(\ell)} &= 2 \log(23),\\
a_1^{(u)} &= (\varrho+1) \frac{2\log(\sqrt{2}+1)+10^{-50}}{p^{(\ell)}} + 2\log y,\\
h^{(u)} &= 2\left(\log p + \log\left(\frac{1}{a_2} + \frac{2}{a_1^{(\ell)} p^{(\ell)}}\right) + \log \lambda + 1.78\right),
\end{align*}
and set $h^{(\ell)}, H^{(\ell)}, \omega^{(u)}, \theta^{(u)}, C^{(u)}, {C'}^{(u)}$ in the same way as Lemma \ref{thm:simplified g bound}. We set $f^{(\ell)}$ in the same way, except the occurrences of $a_1$ are replaced with $a_1^{(u)}$. This puts $g^{(u)}$ in the form of Lemma \ref{lemma3.5} with
\begin{align*}
C_3 &= \log\left(\frac{1}{a_2} + \frac{2}{a_1^{(\ell)} p^{(\ell)}}\right) + \log \lambda + 1.78 + \frac{\lambda}{2\sigma},\\
C_1 &= 1.053 + 2 \sqrt{\omega^{(u)} \theta^{(u)}} C_3 + \log\left(8 {C'}^{(u)} a_2\right),\\
C_2 &= 8 C^{(u)} a_2, \qquad C_4 = (\varrho+1) \frac{2\log(\sqrt{2}+1)+10^{-50}}{2p^{(\ell)}}, \qquad C_5 = 2 \sqrt{\omega^{(u)} \theta^{(u)}}.
\end{align*}

\begin{proof}[Proof of Theorem \ref{thm:improved bound thm 2}]

See the Mathematica notebook \verb|theorem_5_17.nb| of \cite{Katz_Lebesgue-Nagell_Helper_Code} for the computations used in this proof.

To choose the parameters $\mu$ and $\varrho$, again we assume $p$ is set to the implicit function so that $g(23, p(\mu, \varrho), \mu, \varrho) = 0$, and we wish to minimize $p(\mu, \varrho)$. Searching for parameters as before failed, and it appears that the boundary condition $\mu \geq \frac{1}{3}$ is satisfied in the optimal solution. Hence, we instead set $\mu = \frac{1}{3}$ and search for solutions to the following system of two equations in two variables:\[
\left(g\left(23, p, \frac{1}{3}, \varrho\right), \frac{\partial g}{\partial \varrho}\left(23, p, \frac{1}{3}, \varrho\right)\right) = (0, 0).
\]
This yields the solution $(p, \varrho) \approx (1971.41, 22.5978)$. Thus, we set $\mu = \frac{1}{3}, \varrho = 22.5978$, and $p^{(\ell)} = 1973$.

With these choices of parameters, $h^{(\ell)} > 14$ which easily exceeds $\lambda = 2.4249\ldots$ and $\log 2$, and $a_1^{(\ell)} a_2 > 130$, which easily exceeds $\lambda^2$. We also check that
\begin{align}
    (\log p^{(\ell)} + C_3) \log(\log p^{(\ell)} + C_3) = 19.07\ldots > 1.
\end{align}
Then, we find $g^{(u)}$ is negative when we evaluate it with $(y_0, p_0) = (23, 1973)$, so by Lemma \ref{lemma3.5} we see $g^{(u)}$ is negative for $y \geq 23, p \geq 1973$. Therefore, $g(y, p, \mu, \varrho) < 0$ whenever $y \geq 23, p \geq 1973$, so there are no nontrivial solutions when $p \geq 1973$.
\end{proof}

\subsection{Improved bound with Proposition \ref{thm:LaurentThm1}} \label{subsec:improved bound thm 1}

We use the information from Theorem \ref{thm: r is 1} that $r = \pm 1$, and the stronger Proposition \ref{thm:LaurentThm1}, to obtain a sharper upper bound on $p$. To apply this proposition, we must choose the parameters $K, L, R_1, R_2, S_1, S_2 \in \Z_{>0}$ with $K \geq 2$, and $\varrho, \mu, a_1, a_2 \in \R_{>0}$ with $\varrho > 1$ and $\frac{1}{3} \leq \mu \leq 1$. We are subject to three conditions: \eqref{eq:thm1cnd1}, \eqref{eq:thm1cnd2}, and \eqref{eq:thm1cnd3}.

Starting with \eqref{eq:thm1cnd1} and \eqref{eq:thm1cnd2}, we compute the cardinalities as follows.

\begin{lemma} \label{lemma3.6}
With notation as above, for all $R_1, R_2, S_1, S_2 \in \N$, $S_2 \geq 2$,
\begin{align*}
&\#\{\alpha_1^t \alpha_2^s: 0 \leq t < R_1, 0 \leq s < S_1\} = R_1 S_1.\\
&\#\{tb_2 + s b_1: 0 \leq t < R_2, 0 \leq s < S_2\} = (S_2-2)\min(R_2, p) + 2R_2.
\end{align*}
\end{lemma}

\begin{proof}
The first line follows from Lemma \ref{lemma3.4}, since $\alpha_1$ and $\alpha_2$ are multiplicatively independent. For the second set, since $r = \pm 1$, we have $b_1 = p, b_2 = 2$ from \eqref{eq:linear form variables}. Therefore, the second set is
\begin{align*}
\{2t + ps : 0 \leq t < R_2, 0 \leq s < S_2\}.
\end{align*}

We consider two cases. First, if $R_2 \leq p$, then each distinct $(t, s)$ pair gives a different integer, so the cardinality is $R_2 S_2$.

If $R_2 > p$, then when $s$ is even, $2t+sp$ ranges over all even integers from $0$ to $2(R_2-1)+E p$ where $E$ is the largest even integer less than $S_2$. When $s$ is odd, $2t+sp$ ranges over all odd integers from $p$ to $2(R_2-1)+O p$ where $O$ is the largest odd integer less than $S_2$. Thus the cardinality is $R_2 + p E/2 + R_2 + p(O-1)/2 = 2 R_2 + p (E+O-1)/2$. If $S_2$ is even this is $2 R_2 + p(S_2-2+S_2-1-1)/2 = 2 R_2 + (S_2-2)p$. If $S_2$ is odd this is $2 R_2 + p(S_2-1+S_2-2-1)/2 = 2 R_2 + (S_2-2) p$. Thus, either way the cardinality is $2 R_2 + (S_2-2) p$.

In either case, the cardinality is $(S_2-2)\min(R_2, p) + 2 R_2$.
\end{proof}

By Lemma \ref{lemma3.6}, the conditions \eqref{eq:thm1cnd1} and \eqref{eq:thm1cnd2} are equivalent to the conditions
\begin{align}
&R_1 S_1 \geq L, \label{eq16}\\
&(S_2-2)\min(R_2, p) + 2R_2 > (K-1)L.\nonumber
\end{align}

Next, we are subject to the condition
\begin{align*}
a_i &\geq \varrho \abs{\log \alpha_i} - \log \abs{\alpha_i} + 2 D h(\alpha_i),
\end{align*}
which is the same as \eqref{eq:thm2cnd2} without the max with $1$. Hence, by the same reasoning as in Lemma \ref{thm:Thm2 choice of a1 a2 h}, these bounds simplify to
\begin{align*}
a_1 &\geq (\varrho+1) \log \abs{\frac{a+\epsilon b\sqrt{2}}{a-\epsilon b\sqrt{2}}} + 2\log y,\\
a_2 &\geq (\varrho+1) \log(\sqrt{2}+1),
\end{align*}
and, with the same estimate as \eqref{eq14}, assuming $p \geq 100$, we may set
\begin{align*}
a_1 &= (\varrho+1) \frac{2\log(\sqrt{2}+1)+10^{-50}}{p} + 2\log y,\\
a_2 &= (\varrho+1) \log(\sqrt{2}+1).
\end{align*}

With the notation from Proposition \ref{thm:LaurentThm1}, the third constraint \eqref{eq:thm1cnd3} becomes
\begin{align}
K(\sigma L -1) \log \varrho - 3 \log N - 2(K-1) \log b - g L (R a_1 + S a_2) > \epsilon(N). \label{eq17}
\end{align}
If all the constraints are satisfied, Proposition \ref{thm:LaurentThm1} gives the following implicit bound on $\Lambda$:
\begin{align*}
&\abs{\Lambda'} > \varrho^{-\mu K L}, &&\text{where }\Lambda' = \Lambda \max\left\lbrace \frac{LS e^{LS\abs{\Lambda}/4}}{4}, \frac{LR e^{LR\abs{\Lambda}/(2p)}}{2p} \right\rbrace.
\end{align*}
Note that $\Lambda' = \Lambda \max\left(f\left(\frac{LS}{4}\right), f\left(\frac{LR}{2p}\right)\right)$ where $f(x) = x e^{\abs{\Lambda} x}$. Since $f$ is increasing, $\Lambda' = \Lambda f(T)$ where $T = L\max\left(\frac{S}{4}, \frac{R}{2p}\right)$. Then stating the above lower bound logarithmically, it is equivalent to
\begin{align}
\log \abs{\Lambda} + \log T + \abs{\Lambda} T > -\mu KL \log \varrho. \label{eq18}
\end{align}
We summarize these results in the following proposition.

\begin{proposition}\label{thm:thm 1 lower bound}
Suppose we have a nontrivial solution of $x^2-2 = y^p$, $p \geq 100$, with the linear form $-\epsilon\Lambda = b_2 \log \alpha_2 - b_1 \log \alpha_1$ set as in \eqref{eq:linear form variables}. Let $K, L, R_1, R_2, S_1, S_2 \in \N, \varrho, \mu \in \R, K \geq 2, \varrho > 1, \frac{1}{3} \leq \mu \leq 1$, $S_2 \geq 2$, and set $a_1, a_2, T$ as above and $R, S, N, g, \sigma, b, \epsilon(N)$ as in Proposition \ref{thm:LaurentThm1}. If
\begin{align*}
&R_1 S_1 \geq L,\\
&(S_2-2)\min(R_2, p) + 2R_2 > (K-1)L,\\
&K(\sigma L -1) \log \varrho - 3 \log N - 2(K-1) \log b - g L (R a_1 + S a_2) > \epsilon(N),
\end{align*}
then
\begin{align*}
\log \abs{\Lambda} + \log T + \abs{\Lambda} T > -\mu KL \log \varrho.
\end{align*}
\end{proposition}

\begin{remark}\label{rmk:solving for S1 and S2}
Note that from condition \eqref{eq16}, we have $R_1 S_1 \geq L, R_2 S_2 > (K-1)L$, and from this and the fact that $R_1, R_2, S_1, S_2 \geq 1$ one can show $RS \geq KL$, so $g$ is always between $\frac{1}{6}$ and $\frac{1}{4}$. In particular, $g$ is nonnegative, which shows that constraint \eqref{eq17} is easier to satisfy when $a_1$ and $a_2$ are smaller, so it is best to set them to their lower bound as we have done. Similarly, one sees that other than condition \eqref{eq16}, both condition \eqref{eq17} and the lower bound \eqref{eq18} are better when $R_1, S_1, R_2, S_2$ are smaller (it seems backwards for \eqref{eq18} since we are trying to derive a contradiction), so in the optimal solution it must not be possible to decrease them while still satisfying \eqref{eq16}. Thus, one may solve for $S_1$ and $S_2$ in terms of $R_1, R_2, K, L, p$, which reduces the number of variables by $2$.
\end{remark}

Notice that inequality \eqref{eq17} is of the form $A \log \varrho - B \varrho > C$, where
\begin{align*}
A &= K(\sigma L-1),\\
B &= gL\left(\frac{2\log(\sqrt{2}+1)+10^{-50}}{p} R + \log(\sqrt{2}+1) S\right),\\
C &= \epsilon(N) + 3\log N + 2(K-1) \log b\\
&+ gL\left(\left(\frac{2\log(\sqrt{2}+1)+10^{-50}}{p} + 2\log y\right)R + \log(\sqrt{2}+1)S\right).
\end{align*}
The derivative with respect to $\varrho$ of the left-hand side is $\frac{A}{\varrho} - B$, which is decreasing in $\varrho$ assuming $L \geq 2$ (which can be assumed since otherwise the left-hand side of \eqref{eq17} is negative). Therefore, $A\log \varrho - B \varrho$ is concave, so there is at most one interval in $\varrho$ on which \eqref{eq17} is satisfied assuming all other variables are fixed. The only two places where $\varrho$ shows up are \eqref{eq17} and \eqref{eq18}. Thus, it is best when $\varrho$ is smaller, so $\varrho$ should be set to the lower bound of the interval on which $A \log \varrho - B \varrho > C$, if it exists.

To determine if a suitable value of $\varrho$ exists, we find the maximum of $A \log \varrho - B \varrho$, which occurs at $\varrho_0 = \frac{A}{B}$. If $A \log \varrho_0-B \varrho_0 \leq C$, we cannot satisfy \eqref{eq17} with the choices of the other variables. If $\varrho_0 < 1$, then whenever $\varrho \geq 1$, $A \log \varrho - B \varrho \leq -B < 0$, so we cannot satisfy \eqref{eq17}. Otherwise, there is exactly one value of $\varrho$ between $1$ and $\varrho_0$ at which $A \log \varrho - B \varrho = C$, and this is the optimal choice.

Thus, we have reduced the problem to choosing $K, L, R_1, R_2, \mu$, such that the constraints are automatically satisfied assuming there exists any suitable value of $\varrho$.

We note that the factorial appearing in the definition of $\epsilon(N)$ and the superfactorial in the definition of $b$ can be made easier to deal with by replacing them with good approximations from their asymptotic expansions. Specifically, we want an upper bound on $N!$ and a lower bound on $\prod_{k=1}^{K-1} k!$. This is provided by the following proposition.

\begin{proposition}\label{thm:asymptotic expansions}
For $N \in \Z_{>0}$,
\begin{align}
\log(N!) &< \left(N+\frac{1}{2}\right) \log N - N + \frac{1}{2} \log(2\pi) + \frac{1}{12N}. \label{eq19}
\end{align}
With $\epsilon(N)$ set as in Proposition \ref{thm:LaurentThm1}, we have
\begin{align}
\epsilon(N) \leq \epsilon^{(u)}(N) := \frac{3 \log N + \log(2\pi) + \frac{1}{6N} + 2\log\left(1+\left(1-\frac{1}{e}\right)^N\right)}{N}. \label{eq21}
\end{align}
For $K \in \Z_{>0}$, if $A$ we let denote the Glaisher-Kinkelin constant, then
\begin{align}
\log\left(\prod_{k=1}^{K-1} k!\right) &> \left(\frac{K^2}{2}-\frac{1}{12}\right) \log K - \frac{3}{4} K^2 + \frac{K}{2} \log(2\pi) + \frac{1}{12} - \log A - \frac{1}{240 K^2}. \label{eq20}
\end{align}
\end{proposition}
\begin{proof}
We first quote the classical Stirling's formula (see, for example, \cite[Section 6.3]{Edw1974}),
\begin{align*}
\log(N!) &= \left(N+\frac{1}{2}\right) \log N - N + \frac{1}{2} \log(2\pi) + \sum_{i=1}^{m-1} \frac{B_{2i}}{2i(2i-1) N^{2i-1}} + R_m(N),
\end{align*}
where $B_{2i}$ are the Bernoulli numbers and $R_m(N)$ has the same sign as the first term of the asymptotic expansion omitted, that is, $(-1)^{m+1}$. Setting $m = 2$, we obtain \eqref{eq19} (this also follows directly from \cite{Rob1955}). Plugging \eqref{eq19} into the definition of $\epsilon(N)$, we obtain \eqref{eq21} (this is also noted in \cite[p. 342]{Laurent2008}).

Second, we use the asymptotic expansion of the superfactorial, as it relates to the Barnes $G$-function, from \cite{Nemes2014}. Taking $m=2$, and using the fact that the error term in the expansion has the same sign as the first omitted term (see \cite[Theorem 1.2]{Nemes2014}), we have
\begin{align*}
\log\left(\prod_{k=1}^{K-1} k!\right) &> \frac{1}{4} K^2 + K \log K! - \left(\frac{1}{2} K(K+1) + \frac{1}{12}\right) \log K - \log A\\
&+ \sum_{i=1}^{m-1} \frac{B_{2i+2}}{2i(2i+1)(2i+2)K^{2i}} -\frac{1}{720K^2}.
\end{align*}
We then expand $\log K!$ using Stirling's formula with $m=3$, dropping the remainder term for a lower bound. Some simplification then gives the desired result.
\end{proof}

\begin{remark}\label{rmk:epsilon decreasing}
From \eqref{eq21} it follows that $\epsilon(N) \to 0$ as $N \to \infty$. Note that every term is monotonically decreasing in $N$ except $\frac{3 \log N}{N}$, which is decreasing when $\frac{1-\log N}{N^2} < 0$, which is true if and only if $N > e$.
\end{remark}

\subsection{Applying the bounds for large $y$}\label{subsec:applying bounds for large y}

We now give a choice of parameters that depends on $y$ so that we may apply Proposition \ref{thm:thm 1 lower bound} for all sufficiently large $y$. We give a heuristic justification for why this choice is optimal, but make no attempt to prove it rigorously. 

In order to satisfy the bound \eqref{eq17}, the single positive term $K(\sigma L - 1) \log \varrho$ must be large enough to counteract the negative terms. Specifically, since the term $gL(R a_1 + S a_2)$ has $a_1$ in it which grows as $2\log y$, the positive term must also grow as a constant times $\log y$. If $L$ or $R$ are large, this only amplifies the $a_1$ term, so it seems that the optimal solution is to leave $L$ and $R$ as constants. If $\varrho \to \infty$ that increases $a_1$ and $a_2$ linearly in $\varrho$, while only increasing the positive term logarithmically in $\varrho$, so that would be counterproductive. Therefore, it seems that we are forced to have $K$ grow like a constant times $\log y$. Therefore, we set $K = \ceil{K' \log y}$ for some constant $K'$, and keep $L, R_1, R_2, \mu$ constant. Then, as indicated in Remark \ref{rmk:solving for S1 and S2}, we may set $S_1 = \ceil{\frac{L}{R_1}}, S_2 = \ceil{\frac{(K-1)L+1-2R_2}{\min(R_2, p)}}+2$, and condition \eqref{eq16} is satisfied. In order to ensure $S_2 \geq 2$, we impose the mild condition
\begin{align}\label{eq:upp bound R2 less than KL}
    2R_2 \leq (K-1)L + 1.
\end{align}
This condition is easily satisfied when $y$ is large, since we will choose $R_2$ to be a constant and $K \asymp \log y$.

As was the case with our initial bound, we shall find it helpful to replace some of the variables that arise with appropriate bounds as $y \to \infty$. Specifically, in addition to the upper bound $\epsilon^{(u)}(N)$ from \eqref{eq21}, we want upper bounds on the quantities $S$, $b$, and $g$ from Proposition \ref{thm:LaurentThm1}.

\begin{proposition}\label{thm:bounds on S b g}
We have the bounds $S < S^{(u)}, \log b < \log b^{(u)}, g < g^{(u)}$ where
\begin{align*}
S^{(u)} &:= C_1 + C_2 \log y,\\
\log b^{(u)} &:= \log\left(\frac{C_3}{K} + C_4\right) + \frac{3}{2},\\
g^{(u)} &:= \frac{1}{4} - \frac{C_5 K}{C_6 + K} = \frac{1}{4} - C_5 + \frac{C_5 C_6}{C_6 + K},
\end{align*}
and
\begin{align*}
C_1 &= S_1 + 2 + \frac{1-2R_2}{\min(R_2, p)},\\
C_2 &= \frac{K'L}{\min(R_2, p)},\\
C_3 &= \frac{1}{2} \left(2(R-1) + \left(S_1 + \frac{1-L-2R_2}{\min(R_2, p)} + 1\right)p\right),\\
C_4 &= \frac{Lp}{2\min(R_2, p)},\\
C_5 &= \frac{\min(R_2, p)}{12R},\\
C_6 &= \frac{(S_1+2) \min(R_2, p) + 1-L-2R_2}{L}.
\end{align*}
\end{proposition}
\begin{proof}
For $S$, note that
\begin{align}
S &= S_1 + \ceil{\frac{(K-1)L+1-2R_2}{\min(R_2, p)}}+1\nonumber\\
&< S_1 + \frac{(K-1)L+1-2R_2}{\min(R_2, p)} + 2\label{eq22}\\
&< S_1 + \frac{K' L \log y + 1 - 2R_2}{\min(R_2, p)} + 2\nonumber\\
&= S_1 + 2 + \frac{1-2R_2}{\min(R_2, p)} + \frac{K'L}{\min(R_2, p)} \log y,\nonumber
\end{align}
which is $S^{(u)}$.

For $b$, note that by \eqref{eq20}, we have
\begin{align*}
\log b &< \log(2(R-1) + (S-1)p) - \log 2\\
&-\frac{2}{K^2-K} \left(\left(\frac{K^2}{2} - \frac{1}{12}\right) \log K - \frac{3}{4} K^2 + \frac{1}{2} \log(2\pi) K + \frac{1}{12} - \log A - \frac{1}{240K^2}\right)\\
&< \log(2(R-1)+(S-1)p)-\log 2 - \frac{2}{K^2} \left(\frac{K^2}{2} \log K - \frac{3}{4} K^2 + f(K)\right)\\
&= \log(2(R-1)+(S-1)p) - \log 2 - \log K + \frac{3}{2} - \frac{2}{K^2} f(K),
\end{align*}
where $f(K) = \frac{K}{2} \log(2\pi) - \frac{1}{12} \log K + \frac{1}{12} - \log A - \frac{1}{240 K^2}$ consists of asymptotically unimportant terms. Now, note that $f'(K) = \frac{\log(2\pi)}{2} - \frac{1}{12 K} + \frac{1}{120 K^3}$ which is always positive when $K \geq 1$, so $f$ is increasing. Also, $f(1) > 0$, so $f(K)$ is always positive, so we may drop it for an upper bound. Thus, using \eqref{eq22},
\begin{align*}
\log b &< \log(2(R-1) + (S-1)p) - \log(2K) + \frac{3}{2}\\
&< \log\left(2(R-1) + \left(S_1 + \frac{(K-1)L+1-2R_2}{\min(R_2, p)}+1\right)p\right) - \log(2K) + \frac{3}{2}\\
&= \log\left\{\frac{1}{2K}\left(2(R-1) + \left(S_1 + \frac{1-L-2R_2}{\min(R_2, p)}+1\right)p\right) + \frac{Lp}{2\min(R_2, p)}\right\} + \frac{3}{2},
\end{align*}
which is $\log b^{(u)}$. In fact, $\log b$ converges to $\log C_4 + \frac{3}{2}$ as $K \to \infty$, but we only need the upper bound.

For $g$, note that, by \eqref{eq22},
\begin{align*}
g &= \frac{1}{4} - \frac{LK}{12RS}\\
&< \frac{1}{4} - \frac{LK}{12R\left(S_1 + \frac{(K-1)L+1-2R_2}{\min(R_2, p)}+2\right)}\\
&= \frac{1}{4} - \frac{LK}{12R\left(S_1 + \frac{1-L-2R_2}{\min(R_2, p)} + 2 + \frac{KL}{\min(R_2, p)}\right)},
\end{align*}
which is $g^{(u)}$. Again, $g$ converges to $\frac{1}{4}-C_5$ as $K \to \infty$, but we do not need this fact.
\end{proof}

In addition to the constants in Proposition \ref{thm:bounds on S b g}, write $a_1 = C_7 + 2\log y$, where
\begin{align*}
C_7 &= (\varrho+1) \frac{2\log(\sqrt{2}+1)+10^{-50}}{p}.
\end{align*}
Recall also the definition of $\epsilon^{(u)}(N)$ in \eqref{eq21}.

\begin{proposition}
   Assume the notation above. Assuming
    \begin{align}
    &\left(K'(\sigma L - 1) \log \varrho - 2 K' \log b^{(u)} - g^{(u)}L(2R + C_2 a_2)\right) \log y\label{eq23}\\
    &-3 \log(L(K' \log y + 1)) - g^{(u)}L(R C_7 + C_1 a_2) > \epsilon^{(u)}(N)\nonumber
    \end{align}
    and $\frac{S}{4} \geq \frac{R}{2p}$, if $y$ and $p$ are in a solution to \eqref{eq:main lebesgue nagell eqn}, then
    \begin{align}
    &\left(\mu L K' \log \varrho - p/2\right)\log y + 1.053 + \log\left(L \frac{C_1 + C_2 \log y}{4}\right)\label{eq25}\\
    &+ 0.7165 L (C_1 + C_2 \log y) y^{-\frac{p}{2}} + \mu L \log \varrho > 0.\nonumber
    \end{align}
    Therefore, if
    \begin{align}
    K'(\sigma L - 1) \log \varrho - 2 K' \left(\log C_4 + \frac{3}{2}\right) - \left(\frac{1}{4}-C_5\right) L(2R + C_2 a_2) > 0 \label{eq24}
    \end{align}
    and
    \begin{align}
    \mu L K' \log \varrho < p/2,\label{eq26}
    \end{align}
    then there are no solutions to \eqref{eq:main lebesgue nagell eqn} for all sufficiently large $y$ for the given value of $p$.
\end{proposition}
\begin{proof}
Recall constraint \eqref{eq17}, which with our choice of $K$ and $a_1$ substituted in, is
\begin{align*}
&\ceil{K' \log y}(\sigma L - 1) \log \varrho - 3 \log(L \ceil{K' \log y})\\
&- 2(\ceil{K' \log y} - 1) \log b - gL(R (C_7 + 2 \log y) + S a_2) > \epsilon(N).
\end{align*}
This holds if the following inequality is satisfied:
\begin{align*}
&K' (\log y) (\sigma L - 1) \log \varrho - 3 \log(L(K' \log y + 1))\\
&- 2 K' (\log y) (\log b^{(u)}) - g^{(u)} L (R(C_7 + 2 \log y) + S^{(u)} a_2) > \epsilon^{(u)}(N),
\end{align*}
and this is equivalent to \eqref{eq23} after grouping $\log y$ terms.

Clearly, \eqref{eq23} holds as $y \to \infty$ if and only if the coefficient of $\log y$ is positive. Thus, constraint \eqref{eq17} is implied by \eqref{eq24} in the limit as $y \to \infty$.

Note that the lower bound \eqref{eq18} is
\begin{align*}
\log \Lambda + \log T + T \Lambda > -\mu \ceil{K' \log y} L \log \varrho.
\end{align*}
Recall $T = L \max\left(\frac{S}{4}, \frac{R}{2p}\right)$. Now, $\frac{R}{2p}$ is a constant, while $S \sim C_2 \log y$, so if $y$ is large then $S$ will be the maximum. Thus, assuming $y$ is sufficiently large so this is the case, and combining with the upper bounds \eqref{eq8} and \eqref{eq35}, we have
\begin{align*}
1.053 - \frac{p}{2} \log y + \log\left(L \frac{S^{(u)}}{4}\right) + L \frac{S^{(u)}}{4} 2.866 y^{-\frac{p}{2}} > - \mu L (K'\log y+1) \log \varrho.
\end{align*}
Grouping the $\log y$ terms, this is equivalent to \eqref{eq25}.

We obtain a contradiction for all sufficiently large $y$ if and only if the coefficient of $\log y$ in \eqref{eq25} is negative. Thus, to obtain a contradiction we need \eqref{eq26}.
\end{proof}

\begin{remark}
As was the case for \eqref{eq17}, the inequality \eqref{eq24} is of the form $A \log \varrho - B \varrho > C$, where
\begin{align*}
A &= K' (\sigma L - 1)\\
B &= \left(\frac{1}{4}-C_5\right) L C_2 \log(\sqrt{2}+1)\\
C &= 2 K' \left(\log C_4 + \frac{3}{2}\right) +  \left(\frac{1}{4}-C_5\right) L (2R +  C_2 \log(\sqrt{2}+1))
\end{align*}
so we may determine if a suitable value of $\varrho$ exists as before.
\end{remark}

Thus, to optimize the bound on $p$ for the limit as $y \to \infty$, we must choose $K', L, R_1, R_2, \mu$ so that there exists a suitable value of $\varrho$ so that \eqref{eq24} and \eqref{eq26} are satisfied for the minimal value of $p$. (We also need that \eqref{eq:upp bound R2 less than KL} is satisfied, but this is automatic for sufficiently large $y$.) We searched the parameter space using code written in Python (see the file \verb|parameter_search_for_theorem_1_4.py| of \cite{Katz_Lebesgue-Nagell_Helper_Code}) to find the minimal value of $\mu L K' \log \varrho$ for each value of $p$, and the lowest value of $p$ for which it was less than $p/2$ was $p=916$ with $(K', L, R_1, R_2, \mu) = (26.62, 9, 1, 64, 0.58)$. This sets the limit of Theorem \ref{thm:only triv solns for p big}.

We know now what happens when $y$ is sufficiently large. Our next step is to make these results explicit with respect to the size of $y$. That is, we find an explicit lower bound on $y$ for which \eqref{eq23} is satisfied and \eqref{eq25} is not satisfied with a given choice of parameters satisfying \eqref{eq24} and \eqref{eq26}.

\begin{proposition} \label{proposition3.8}
Let $b_0, g_0, \epsilon_0$ be the values of $b^{(u)}, g^{(u)}, \epsilon^{(u)}(N)$, respectively at a particular value $y_0$ of $y$. Put
\begin{align*}
C_8 &= K'(\sigma L - 1) \log \varrho - 2 K' \log b_0 - g_0L(2R + C_2 a_2),\\
C_9 &= 3 \log L + g_0 L(R C_7 + C_1 a_2) + \epsilon_0,
\end{align*}
and suppose
\begin{align}
C_8 \log y_0 - 3 \log (K' \log y_0+1) > C_9 \label{eq27}
\end{align}
is satisfied for a particular choice of $K', L, R_1, R_2, \mu, \varrho$. Assuming $C_3, C_6, C_8 > 0$, $\log y_0 > \frac{3}{C_8} - \frac{1}{K'}$, $N > e$, then \eqref{eq23} is satisfied for all $y \geq y_0$ with the same choice of $K', L, R_1, R_2, \mu, \varrho$.
\end{proposition}
\begin{proof}
Clearly $\log b^{(u)}$ and $g^{(u)}$ are decreasing in $K$ when $C_3, C_6 > 0$, and, as mentioned in Remark \ref{rmk:epsilon decreasing}, $\epsilon^{(u)}(N)$ is decreasing when $N > e$, so we may substitute them for their values at $y_0$ as an upper bound. After the substitution, \eqref{eq23} becomes
\begin{align*}
    C_8 \log y - 3 \log(K' \log y + 1) > C_9.
\end{align*}
The derivative of the left-hand side with respect to $\log y$ is $C_8 - \frac{3K'}{K' \log y + 1}$, which is increasing, and positive when $\log y > \frac{3}{C_8} - \frac{1}{K'}$, assuming $C_8 > 0$. Thus, if \eqref{eq27} is satisfied at $y_0$, then it is satisfied for all $y \geq y_0$.
\end{proof}

\begin{proposition} \label{proposition3.9}
Suppose \eqref{eq25} is not satisfied for a particular choice of parameters at a particular value $y_0$ of $y$. Also assume $C_1 > 0$ and
\begin{align}
\mu L K' \log \varrho - \frac{p}{2} + \frac{C_2}{C_1 + C_2 \log y_0} + 0.7165 L C_2 y_0^{-\frac{p}{2}} < 0. \label{eq29}
\end{align}
Then \eqref{eq25} is not satisfied for all $y \geq y_0$.
\end{proposition}
\begin{proof}
The derivative of \eqref{eq25} with respect to $\log y$ is
\begin{align*}
\mu L K' \log \varrho - \frac{p}{2} + \frac{C_2}{C_1 + C_2 \log y} + 0.7165 L\left(-\frac{p}{2} (C_1 + C_2 \log y) + C_2\right) y^{-\frac{p}{2}}.
\end{align*}
This is negative if $C_1 > 0$ and
\begin{align*}
    \mu L K' \log \varrho - \frac{p}{2} + \frac{C_2}{C_1 + C_2 \log y} + 0.7165 L C_2 y^{-\frac{p}{2}} < 0.
\end{align*}
Since the left-hand side is decreasing in $y$, if it is negative at $y_0$, then it is negative for all $y \geq y_0$.
\end{proof}

After searching for solutions to \eqref{eq24} and \eqref{eq26} that minimized $\mu L K'\log \varrho$ in the limit as $y \to \infty$, we increased $\varrho$ slightly to satisfy \eqref{eq27} at a smaller value of $y$ while still not satisfying \eqref{eq25}. Through this process, we found the choices of parameters in Table \ref{Tab:parameters} below to satisfy the hypotheses of Propositions \ref{proposition3.8} and \ref{proposition3.9} for the primes between $916$ and $1951$:

\begin{table}[ht]
\caption{Values of parameters for $919 \leq p \leq 1951$}
\centering
\begin{tabular}{|L|L|L|L|L|L|L|L|}
\hline
p & y & K' & L & R_1 & R_2 & \mu & \varrho\\
\hline
919 & 10^{800} & 26.64 & 9 & 1 & 64 & 0.58 & 27.22\\
\hline
929 & 10^{300} & 26.71 & 9 & 1 & 64 & 0.58 & 27.8\\
\hline
937, 941 & 10^{150} & 26.76 & 9 & 1 & 64 & 0.58 & 28.55\\
\hline
947 & 10^{100} & 26.83 & 9 & 1 & 64 & 0.58 & 29.3\\
\hline
953 & 10^{100} & 26.42 & 9 & 1 & 64 & 0.59 & 29.8\\
\hline
967-997 & 10^{100} & 26.67 & 9 & 1 & 64 & 0.59 & 29.8\\
\hline
1000-1200 & 10^{100} & 27.04 & 10 & 1 & 64 & 0.57 & 26.3\\
\hline
1200-1951 & 10^{50} & 28.69 & 10 & 1 & 69 & 0.59 & 33\\
\hline
\end{tabular}
\label{Tab:parameters}
\end{table}

See the file \verb|theorem_5_1.sage| of \cite{Katz_Lebesgue-Nagell_Helper_Code} for the code verifying these choices of parameters. This concludes the proof of Theorem \ref{thm:no nontrivial solutions p and y big}.

\section{Ruling out small $y$}\label{sec:small y}

By Theorem \ref{thm:no nontrivial solutions p and y big}, we know any nontrivial solutions to \eqref{eq:main lebesgue nagell eqn} with $p > 911$ must have $y$ smaller than the values in Table \ref{Tab:parameters}. In order to finish the proof of Theorem \ref{thm:only triv solns for p big}, we must show there are no nontrivial solutions less than these bounds. In order to accomplish this, we return to studying the Thue equation \eqref{eq4}. By Theorem \ref{thm: r is 1}, we may assume $r = \pm 1$. As mentioned in Section \ref{sec:ClassicalApproach}, we may then assume $r=1$, at the cost of losing control of the sign of $x$. Let $f = f_{1,p}$ be the polynomial corresponding to this equation, and $\theta, \rho_i$ be as in Section \ref{sec:Galois Theory}.

Let $a$ and $b$ be a solution to the Thue equation \eqref{eq4}. If $b=0$, then clearly the Thue equation implies $a=1$, so we have the trivial solution to the Thue equation. Therefore, we may assume $|b| \geq 1$, so we have a \emph{nontrivial} solution to \eqref{eq4}. By factoring $f$, we obtain
\begin{align}
\left(\frac{a}{b}-\theta\right) \left(\frac{a}{b}-\rho_1\right) \dots \left(\frac{a}{b}-\rho_{p-1}\right) &= \frac{1}{b^p}. \label{eq30}
\end{align}
Since $\theta$ is the only real root of $f$ (Theorem \ref{thm:formula for roots}), the absolute value of the factor $\frac{a}{b}-\rho_i$ for $i \neq 0$ is bounded below by the absolute value of the imaginary part of $\rho_i$.

\begin{theorem}\label{thm:product of imag rho_i}
We have
\begin{align*}
\prod_{i=1}^{p-1} \abs{\im(\rho_i)} \geq p 2^{\frac{p-3}{2}}.
\end{align*}
\end{theorem}
\begin{proof}
Recall, from \eqref{eq503} with $r=1$,
\begin{align*}
\rho_i = \sqrt{2} - \frac{2\sqrt{2}}{(1+\sqrt{2})^{-2/p} \zeta_p^i+1}.
\end{align*}
Thus,
\begin{align*}
\rho_i = \sqrt{2} - 2\sqrt{2} \frac{(1+\sqrt{2})^{-2/p} \zeta_p^{-i}+1}{\abs{(1+\sqrt{2})^{-2/p} \zeta_p^i+1}^2},
\end{align*}
and
\begin{align*}
\im(\rho_i) &= 2\sqrt{2} \frac{(1+\sqrt{2})^{-2/p} \sin\left(\frac{2\pi i}{p}\right)}{\abs{(1+\sqrt{2})^{-2/p} \zeta_p^i+1}^2}.
\end{align*}
Therefore,
\begin{align*}
\prod_{i=1}^{p-1} \im(\rho_i) &= \left(2\sqrt{2}(1+\sqrt{2})^{-2/p}\right)^{p-1} \frac{\prod_{i=1}^{p-1} \sin\left(\frac{2\pi i}{p}\right)}{\prod_{i=1}^{p-1} \abs{(1+\sqrt{2})^{-2/p} \zeta_p^i+1}^2}.
\end{align*}
The polynomial $(x-1)^p-(1+\sqrt{2})^{-2}$ has roots $(1+\sqrt{2})^{-2/p} \zeta_p^i + 1$ for $0 \leq i \leq p-1$, so
\begin{align*}
&\prod_{i=0}^{p-1} \left((1+\sqrt{2})^{-2/p} \zeta_p^i + 1\right) = (1+\sqrt{2})^{-2}+1,
\end{align*}
and hence
\begin{align*}
    \prod_{i=1}^{p-1} \abs{(1+\sqrt{2})^{-2/p} \zeta_p^i+1} = \frac{(1+\sqrt{2})^{-2}+1}{(1+\sqrt{2})^{-2/p}+1}.
\end{align*}
By rearranging and relating the product over roots of unity to the value of a cyclotomic polynomial, we find
\begin{align*}
\prod_{j=1}^{p-1} \sin\left(\frac{2\pi j}{p}\right) &= \prod_{j=1}^{p-1} \frac{e^{2\pi i \frac{j}{p}}-e^{-2\pi i \frac{j}{p}}}{2i}= (2i)^{1-p} \prod_{j=1}^{p-1} e^{-2\pi i \frac{j}{p}} \left(e^{2\pi i \frac{2j}{p}}-1\right)\\
&= (2i)^{1-p} e^{-2\pi i \frac{p(p-1)}{2p}} \prod_{k=1}^{p-1} \left(e^{2\pi i \frac{k}{p}}-1\right) = (-1)^{\frac{p-1}{2}} \frac{p}{2^{p-1}}.
\end{align*}
Putting this all together, we have,
\begin{align*}
\prod_{i=1}^{p-1} \abs{\im(\rho_i)} &= \left(2\sqrt{2}(1+\sqrt{2})^{-2/p}\right)^{p-1} \cdot \frac{p}{2^{p-1}} \cdot \left(\frac{(1+\sqrt{2})^{-2/p}+1}{(1+\sqrt{2})^{-2}+1}\right)^2\\
&= p\left(\sqrt{2} (1+\sqrt{2})^{-2/p}\right)^{p-1} \left(\frac{(1+\sqrt{2})^{-2/p}+1}{(1+\sqrt{2})^{-2}+1}\right)^2\\
&= p 2^{\frac{p-1}{2}} \left(\frac{(1+\sqrt{2})^{-1/p}+(1+\sqrt{2})^{1/p}}{(1+\sqrt{2})^{-1}+1+\sqrt{2}}\right)^2\\
&= p 2^{\frac{p-1}{2}} \frac{((1+\sqrt{2})^{-1/p}+(1+\sqrt{2})^{1/p})^2}{8}.
\end{align*}
Since $x+\frac{1}{x} \geq 2$ for $x>0$, we have
\begin{align*}
\prod_{i=1}^{p-1} \abs{\im(\rho_i)} &\geq p 2^{\frac{p-3}{2}}. \qedhere
\end{align*}
\end{proof}

Combining Theorem \ref{thm:product of imag rho_i} with \eqref{eq30}, we have
\begin{align}
\abs{\frac{a}{b}-\theta} \leq \frac{C_p}{\abs{b}^p}, \label{eq31}
\end{align}
where $C_p = p^{-1} 2^{\frac{3-p}{2}}$. Note that $C_p < \frac{1}{2}$ for $p \geq 3$, so $\abs{\frac{a}{b}-\theta} < \frac{1}{2\abs{b}^p} \leq \frac{1}{2b^2}$, and therefore $\frac{a}{b}$ is a convergent to $\theta$ (see \cite[Theorem 19]{khinchin1964continued}). Inequality \eqref{eq31} implies $\frac{a}{b}$ is a much better approximation to $\theta$ than we would expect, since by Roth's theorem \cite{Roth1955} there are only finitely many solutions to $\abs{\frac{a}{b}-\theta} < |b|^{-2-\epsilon}$ for any fixed $\epsilon > 0$.

\begin{lemma}\label{thm:ratio is small}
    Let $a$ and $b$ be a nontrivial solution to the Thue equation \eqref{eq4}. Then $\abs{\frac{a}{b}} < 0.1$.
\end{lemma}

\begin{proof}
By Theorem \ref{thm:thue equations small p}, we may assume $p \geq 17$. Then
\begin{align*}
\theta &= \sqrt{2} - \frac{2\sqrt{2}}{(1+\sqrt{2})^{-2/p} + 1},
\end{align*}
which is monotonically increasing in $p$ to $0$. Plugging in $p=17$, we obtain $\abs{\theta} < 0.08$ for $p \geq 17$. It will be shown later in Theorem \ref{thm:observation2} that $b$ is even, so we may assume $\abs{b} \geq 2$. Since $C_p < \frac{1}{2}$, we have $\abs{\frac{a}{b}} \leq \abs{\theta} + \abs{\frac{a}{b}-\theta} < 0.08 + \frac{1}{2^{p+1}} < 0.1$.
\end{proof}

\begin{lemma}\label{thm:bound for y from b}
    Let $x^2-2 = y^p$ be a nontrivial solution to \eqref{eq:main lebesgue nagell eqn} with $a, b$ as in \eqref{eq5} and \eqref{eq4}. Then $y > 1.99 b^2$.
\end{lemma}

\begin{proof}
    By Lemma \ref{lemma3.1} and Lemma \ref{thm:ratio is small},
    \begin{align*}
    y &= 2b^2-a^2= b^2\left(2-\left(\frac{a}{b}\right)^2\right)> b^2\left(2-0.1^2\right). \qedhere
    \end{align*}
\end{proof}

Let the continued fraction expansion of $\theta$ be $[q_0; q_1, q_2, \dots]$ (by Proposition \ref{prop:real roots}, $-\sqrt{2} < \theta < 0$, so $q_0 < 0$). Denote the $k$th convergent to $\theta$ by $\frac{P_k}{Q_k}$, so that we have the usual recurrence relations
\begin{align*}
&P_{-2} = 0, &&P_{-1} = 1, &&P_{k+1} = P_{k-1} + q_{k+1} P_k,\\
&Q_{-2} = 1, &&Q_{-1} = 0, &&Q_{k+1} = Q_{k-1} + q_{k+1} Q_k.
\end{align*}

\begin{proposition}\label{prop:cont frac lower bound}
    Let $p \geq 17$ be a prime. With notation as above, suppose $q_{i+1} \leq p 2^{\frac{3p-7}{2}}-2$ for $1 \leq i \leq k$. If $a$ and $b$ give a nontrivial solution to \eqref{eq4}, then $\abs{b} \geq Q_{k+1}$.
\end{proposition}
\begin{proof}
By \cite[equation (34)]{khinchin1964continued}, the convergents for $k \geq 0$ satisfy\[
\abs{\theta-\frac{P_k}{Q_k}} > \frac{1}{(q_{k+1}+2)Q_k^2}.
\]
Assuming $\frac{P_k}{Q_k} = \frac{a}{b}$, this together with \eqref{eq31} implies
\begin{align*}
&\frac{1}{(q_{k+1}+2)Q_k^2} < \frac{p^{-1} 2^{\frac{3-p}{2}}}{Q_k^p},
\end{align*}
which holds if and only if
\begin{align}
    q_{k+1} > p 2^{\frac{p-3}{2}} Q_k^{p-2}-2.\label{eq32}
\end{align}
As in the proof of Lemma \ref{thm:ratio is small}, we have $\abs{b} \geq 2$. Thus, weakening \eqref{eq32} further still, we must have
\begin{align}
q_{k+1} > p 2^{\frac{p-3}{2}} 2^{p-2}-2 = p 2^{\frac{3p-7}{2}}-2. \label{eq33}
\end{align}
Thus, if $q_{i+1}$ never exceeds this bound for $1 \leq i \leq k$, then $\abs{b} \geq Q_{k+1}$.
\end{proof}

\begin{proposition}\label{thm:bound for a and y from b}
    Let $p \geq 17$ be a prime, and set $\theta$ as above. Let $x, y$ be a nontrivial solution to \eqref{eq:main lebesgue nagell eqn} with $a, b$ as in \eqref{eq5} and \eqref{eq4}. If we have a bound $\abs{b} \geq b_0$, and if $b_0 \geq a_0 / (\abs{\theta}-C_p/2^p)$ and $b_0 \geq \sqrt{y_0/1.99}$, then $\abs{a} \geq a_0$ and $y \geq y_0$.
\end{proposition}

\begin{proof}
    That $|a| \geq a_0$ follows from \eqref{eq31} and the fact that $\abs{b} \geq 2$, so\[
    \abs{\frac{a}{b}} \geq \abs{\theta} - \frac{C_p}{2^p}.
    \]
    Hence, $\abs{a} = \abs{b} \abs{\frac{a}{b}} \geq b_0 (\abs{\theta}-C_p/2^p) \geq a_0$. That $y \geq y_0$ follows from Lemma \ref{thm:bound for y from b}, since $y > 1.99b^2 \geq 1.99b_0^2 \geq y_0$.
\end{proof}

We ran code in Sage to compute the continued fraction expansion of $\theta$ for different primes $p$ and obtain lower bounds via Propositions \ref{prop:cont frac lower bound} and \ref{thm:bound for a and y from b}.

See the file \verb|theorems_1_4_and_1_5_and_6_6.sage| of \cite{Katz_Lebesgue-Nagell_Helper_Code}.

\begin{proof}[Proof of Theorem \ref{thm:only triv solns for p big}]
   In the script mentioned above, we used Proposition \ref{thm:bound for a and y from b} to prove that $y$ exceeds the lower bounds in Table \ref{Tab:parameters} for $919 \leq p \leq 1951$.
\end{proof}

In addition, in the same script we followed the same procedure on the values of $p$ from $17$ to $911$ to prove that $|a|, |b|, y \geq 10^{1000}$. This proves Theorem \ref{thm:lower bound on y}. We also state the result for $a$ and $b$ below.

\begin{theorem}\label{thm:lower bound on a and b}
    Let $(a, b)$ be a nontrivial solution to \eqref{eq4}. Then $\abs{a}, \abs{b} > 10^{1000}$.
\end{theorem}

Further computations could rule out even larger values of $y, |a|, |b|$.

\section{Further observations} \label{sec:other observations}

From the Thue equation \eqref{eq4}, we collect some additional observations on $x, y, a$, and $b$. Throughout this section, we assume that $r=1$ and $a, b$ are a nontrivial solution to \eqref{eq4}.

\begin{theorem}\label{thm:observation1}
    With notation as above, $y = 2b^2-a^2$ and $\gcd(a, b) = 1$.
\end{theorem}
\begin{proof}
This follows from specializing Lemma \ref{lemma3.1} to $r = 1$ (we already used the fact that $y=2b^2-a^2$ in Section \ref{sec:small y}).
\end{proof}

\begin{theorem}\label{thm:observation2}
    With notation as above, $a$ is odd and $b$ is even.
\end{theorem}
\begin{proof}
We have $y=2b^2-a^2$ by Theorem \ref{thm:observation1}. Since $y \equiv -1 \pmod{8}$ by Theorem \ref{thm:elementary 1}, we must have $a$ is odd. Then $-1 \equiv 2b^2-1 \pmod{8}$, so $2b^2 \equiv 0 \pmod{8}$, which implies $b$ is even.
\end{proof}

\begin{theorem}\label{thm:observation3}
    With notation as above, $\sgn(a) = -\sgn(x)$ and $\sgn(b) = \sgn(x)$.
\end{theorem}
\begin{proof}
This follows from Lemma \ref{lemma2} and the fact that we are taking $r$ to be $1$, so if $x$ is negative then we must negate both $x$ and $b$ and take $r$ to be $-1$ in Lemma \ref{lemma2}.
\end{proof}

\begin{theorem}\label{thm:observation4}
    With notation as above, we have the following equalities:
    \begin{align}
        x-1 &= (1+\sqrt{2})((a+b\sqrt{2})^p-1)\label{eq1005}\\
        &= (1+\sqrt{2})(a-1+b\sqrt{2})\left(\sum_{i=0}^{p-1} (a+b\sqrt{2})^i\right)\nonumber\\
        &= \frac{1}{2\sqrt{2}} \left((a+b\sqrt{2})^p-(a-b\sqrt{2})^p\right)\label{eq1003}\\
        &= \sum_{j=0}^{\frac{p-1}{2}} \binom{p}{2j+1} a^{p-2j-1} b^{2j+1} 2^j. \label{eq1000}\\
        x-2 &= -\frac{1}{2} ((a+b\sqrt{2})^p+(a-b\sqrt{2})^p)\label{eq1004}\\
        &= -\sum_{j=0}^{\frac{p-1}{2}} \binom{p}{2j} a^{p-2j} b^{2j} 2^j. \label{eq1001}
    \end{align}
\end{theorem}

\begin{proof}
    By \eqref{eq5} with $r=1$, we have
    \begin{align*}
    x-1 &= x+\sqrt{2}-(1+\sqrt{2})= (1+\sqrt{2})((a+b\sqrt{2})^p-1),
    \end{align*}
    and this gives \eqref{eq1005}. Next, note that\[
    (a+b\sqrt{2})^p = (\sqrt{2}-1)(x+\sqrt{2}) = 2-x + (x-1)\sqrt{2}
    \]
    so by adding and subtracting this equation from its conjugate, we obtain \eqref{eq1003} and \eqref{eq1004}. The other equalities follow from factoring and the binomial theorem.
\end{proof}

\begin{corollary}\label{thm:observation5}
    With notation as above, $b \mid (x-1)$ and $a \mid (x-2)$.
\end{corollary}
\begin{proof}
Every term in \eqref{eq1000} is divisible by $b$, and every term in \eqref{eq1001} is divisible by $a$.
\end{proof}

It is also fruitful to rearrange the Thue equation \eqref{eq4} as follows:
\begin{align}
&\sum_{k=0}^p \binom{p}{k} 2^{\floor{\frac{k}{2}}} a^{p-k} b^k = 1 \iff a^p-1 = -\sum_{k=1}^p \binom{p}{k} 2^{\floor{\frac{k}{2}}} a^{p-k} b^k \nonumber\\
\iff &(a-1)(1 + a + \dots + a^{p-1}) = -b\sum_{k=1}^p \binom{p}{k} 2^{\floor{\frac{k}{2}}} a^{p-k} b^{k-1}. \label{eq1002}
\end{align}

\begin{theorem}\label{thm:observation6}
    With notation as above,
    \begin{enumerate}
        \item $v_p(a-1) = v_p(b)$.
        
        \item If $p \mid b$, then $v_p(x-1) = v_p(b)+1$, and otherwise $p \nmid (x-1)$. In particular, either $p \nmid (x-1)$ or $p^2 \mid (x-1)$ (this agrees with Remark \ref{rmk:elementary 1}).
    	
        \item If $p \mid a$, then $v_p(x-2) = v_p(a)+1$, and otherwise $p \nmid (x-2)$. In particular, either $p \nmid (x-2)$ or $p^2 \mid (x-2)$.
    \end{enumerate}
\end{theorem}

\begin{proof}
    Notice if $p \nmid b$, then the only term of the right-hand side of \eqref{eq1002} not divisible by $p$ is $-2^{\frac{p-1}{2}} b^p$, and so the left-hand side is not divisible by $p$, so $p \nmid (a-1)$. On the other hand, if $p \mid b$, then the $k=1$ term has a lower $p$-adic valuation than any other, so the $p$-adic valuation of the right-hand side is $v_p(-p a^{p-1} b) = v_p(b)+1$. The $p$-adic valuation of the left-hand side is $v_p(a-1)+1$ by Lemma \ref{lemma1}, so $v_p(a-1) = v_p(b)$. Thus, $v_p(a-1) = v_p(b)$ in either case. This proves statement (1).
    
    For statements (2) and (3), we consider expressions \eqref{eq1000} and \eqref{eq1001}, respectively. For instance, if $p \nmid b$, then the only term of \eqref{eq1000} not divisible by $p$ is $b^p 2^{\frac{p-1}{2}}$, so $x-1$ is not divisible by $p$. On the other hand, if $p \mid b$, then since $p \geq 3$, the $p$-adic valuation of each term with $j>0$ is greater than that of the $j=0$ term, so $v_p(x-1) = v_p (p a^{p-1} b) = v_p(b)+1$. Statement (3) is similar.
\end{proof}

\begin{theorem}\label{thm:observation7}
    With notation as above,
    \begin{enumerate}
        \item $x-1 \equiv \left(\frac{2}{p}\right) b \pmod{p}$,
        \item $x-2 \equiv -a \pmod{p}$,
        \item $a-1 \equiv -\left(\frac{2}{p}\right) b \pmod{p}$.
    \end{enumerate}
\end{theorem}

\begin{proof}
    By reducing \eqref{eq1000} modulo $p$, all terms vanish except the $j=\frac{p-1}{2}$ term, and we obtain $x-1 \equiv b^p 2^{\frac{p-1}{2}} \equiv b \left(\frac{2}{p}\right) \pmod{p}$. Similarly, reducing \eqref{eq1001} modulo $p$, we obtain $x-2 \equiv -a^p \equiv -a \pmod{p}$. Combining these congruences yields $a-1 \equiv -\left(\frac{2}{p}\right) b \pmod{p}$, which, in fact, is equivalent to the Thue equation \eqref{eq4} modulo $p$.
\end{proof}

\begin{theorem}\label{thm:observation8}
    With notation as above, $v_2(x-1) = v_2(b) = v_2(a-1)$.
\end{theorem}

\begin{proof}
    Since $b$ is even and $a$ is odd by Theorem \ref{thm:observation2}, the $j=0$ term in \eqref{eq1000} has a lower $2$-adic valuation than any other term, and so $v_2(x-1) = v_2(p a^{p-1} b) = v_2(b)$. Also, the $k=1$ term in \eqref{eq1002} has a lower $2$-adic valuation than any other term, so the $2$-adic valuation of the right-hand side is $v_2(-p a^{p-1} b) = v_2(b)$. Since $1 + a + \dots + a^{p-1} \equiv p \equiv 1 \pmod{2}$, the $2$-adic valuation of the left-hand side is $v_2(a-1)$. Thus, $v_2(a-1) = v_2(b)$.
\end{proof}

\begin{theorem}\label{thm:observation9}
    For every prime $\ell \mid b$ such that $\ell \not\equiv 1 \pmod{p}$, we have $v_\ell(a-1) = v_\ell(b)$.
\end{theorem}
\begin{proof}
    Note that, since $b$ divides every term in the sum on the right-hand side of \eqref{eq1002} except the $k=1$ term, the GCD of the two factors on the right-hand side of \eqref{eq1002} is $\gcd(b, p a^{p-1})$. Since $\gcd(a, b) = 1$ by Theorem \ref{thm:observation1}, this GCD is $p$ or $1$, according to whether $p \mid b$ or not. By Lemma \ref{lemma1}, any prime $\ell \not\equiv 1 \pmod{p}$ with $\ell \neq p$ does not divide $\frac{a^p-1}{a-1}$. Hence, if $\ell \mid b$ and $\ell \neq p$, then since $\ell$ does not divide the other factor in the right-hand side, the $\ell$-adic valuation of the right-hand side is $v_\ell(b)$, so $v_\ell(a-1) = v_\ell(b)$. For $\ell=p$, this was proven in Theorem \ref{thm:observation6}.
\end{proof}

\begin{corollary}\label{thm:observation10}
    There is a prime $\ell \equiv 1 \pmod{p}$ with $\ell \mid b$.
\end{corollary}

\begin{proof}
Assume for contradiction that, for every prime $\ell$ that divides $b$, we have $\ell \not \equiv 1 \pmod{p}$. Then, by Theorem \ref{thm:observation9}, we must have $b \mid (a-1)$. However, since $a \neq 1$, this implies $\abs{a-1} \geq \abs{b}$, so $\abs{\frac{a}{b}-\frac{1}{b}} \geq 1$. Then $\abs{\frac{a}{b}} \geq 1-\frac{1}{\abs{b}} \geq \frac{1}{2}$ since $b$ is nonzero and even (Theorem \ref{thm:observation2}), but this contradicts Lemma \ref{thm:ratio is small}.
\end{proof}

\begin{corollary}\label{thm:observation11}
    With notation as above, $b \equiv \gcd(a-1, b) \pmod{p}$
\end{corollary}

\begin{proof}
    For any prime $\ell \mid b$ such that $\ell \not\equiv 1 \pmod{p}$, by Theorem \ref{thm:observation9} we have $v_\ell(a-1) = v_\ell(b) = v_\ell(\gcd(a-1, b))$, so $\ell \nmid \frac{b}{\gcd(a-1, b)}$. Therefore, $\frac{b}{\gcd(a-1, b)}$ is only divisible by primes that are $1$ modulo $p$, so it is itself $1$ modulo $p$. Thus $b = \gcd(a-1, b) \cdot \frac{b}{\gcd(a-1, b)} \equiv \gcd(a-1, b) \pmod{p}$.
\end{proof}

\begin{theorem}\label{thm:observation12}
    With notation as above, we have $a^p \equiv 1 \pmod{b}$ and $b^p \equiv 2^{\frac{1-p}{2}} \pmod{a}$. Also, $-y$ and $a$ are both primitive $p$th roots of $1$ modulo $b$.
\end{theorem}

\begin{proof}
    From \eqref{eq1002} we see $b \mid (a^p-1)$, so $a^p \equiv 1 \pmod{b}$. Reducing \eqref{eq4} modulo $a$, we see $2^{\frac{p-1}{2}} b^p \equiv 1 \pmod{a}$, so $b^p \equiv 2^{\frac{1-p}{2}} \pmod{a}$.

    As in the proof of Corollary \ref{thm:observation10}, $a \not\equiv 1 \pmod{b}$. Since $a^p \equiv 1 \pmod{b}$, $a$ is a primitive $p$th root of $1$ mod $b$. Note that $-y = a^2-2b^2 \equiv a^2 \pmod{b}$. Since $a$ is a primitive $p$th root of $1$, so is $a^2$, so $-y$ is as well.
\end{proof}

\begin{theorem}\label{thm:observation13}
    We have $\gcd\left(b, \frac{x-1}{b}\right) = 1$ or $p$. For every prime $\ell$ other than $p$ that divides $\frac{x-1}{b}$, we have 
    \begin{align*}
        \ell \equiv 1 \pmod{p}, \ \ \ \ \ell \equiv \pm 1 \pmod{8},
    \end{align*}
    or
    \begin{align*}
        \ell \equiv \pm 1 \pmod{p}, \ \ \ \ \ell \equiv \pm 3 \pmod{8}.
    \end{align*}
Also, $\gcd\left(a, \frac{x-2}{a}\right) = 1$ or $p$, and the same conditions apply to all primes dividing $\frac{x-2}{a}$. In particular, if $\ell \mid b$ is prime and $\ell \not\equiv \pm 1, 0 \pmod{p}$, then $v_\ell(b) = v_\ell(x-1)$. Similarly, if $\ell \mid a$ is prime and $\ell \not\equiv \pm 1, 0 \pmod{p}$, then $v_\ell(a) = v_\ell(x-2)$.
\end{theorem}

\begin{proof}
Recall from Corollary \ref{thm:observation5} that $b \mid(x-1)$. We mimic the proof of Lemma \ref{lemma1}. Suppose $d \mid b$ and $d \mid \frac{x-1}{b}$. Then by \eqref{eq1000}, $0 \equiv \frac{x-1}{b} \equiv p a^{p-1} \pmod{d}$. Since $d$ divides $b$ and $b$ is relatively prime to $a$ (by Theorem \ref{thm:observation1}), we must have that $d \mid p$. Thus, $\gcd\left(b, \frac{x-1}{b}\right)$ is $1$ or $p$. In fact, if the GCD is $p$, then $p \parallel \frac{x-1}{b}$ by Theorem \ref{thm:observation6}.
    
    Suppose $\ell \mid \frac{x-1}{b}, \ell \neq p$ is prime. Then by \eqref{eq1003}, $\ell \mid \frac{(a+b\sqrt{2})^p-(a-b\sqrt{2})^p}{2\sqrt{2}}$, so
    \begin{align}\label{eq36}
    (a+b\sqrt{2})^p \equiv (a-b\sqrt{2})^p \pmod{\ell \, \Z[\sqrt{2}]}.
    \end{align}
    Note that $\ell \neq 2$ by Theorem \ref{thm:observation8}. Let $\F$ denote the field $\F_\ell(\theta)$, where $\theta$ is a square root of $2$. We apply the homomorphism from $\mathbb{Z}[\sqrt{2}] \rightarrow \F$ given by $u+v\sqrt{2}\mapsto u+v\theta$ and find \eqref{eq36} implies $(a+b\theta)^p = (a-b\theta)^p$ in $\F$. In this case, if $p \nmid (\abs{\F}-1)$, then the map $x \mapsto x^p$ is injective on $\F$, hence $a+b\theta = a-b\theta$, so $2b\theta = 0$. Since $2\theta$ is a unit in $\F$, $b \equiv 0 \pmod{\ell}$. This contradicts that $\gcd\left(b, \frac{x-1}{b}\right) \mid p$. Therefore, we have $p \mid (\abs{\F}-1)$, so if $\ell \equiv \pm 1 \pmod{8}$ then $\ell \equiv 1 \pmod{p}$, and if $\ell \equiv \pm 3 \pmod{8}$ then $\ell \equiv \pm 1 \pmod{p}$.

    A similar argument works for $a$ and $\frac{x-2}{a}$. If $d \mid a$ and $d \mid \frac{x-2}{a}$, then by \eqref{eq1001}, $0 \equiv \frac{x-2}{a} \equiv -p b^{p-1} 2^{\frac{p-1}{2}} \pmod{d}$, and since $a$ is odd and relatively prime to $b$, this implies $d \mid p$. If $\ell \mid \frac{x-2}{a}, \ell \neq p$ is prime, then since by \eqref{eq1004}, $\ell \mid \frac{(a+b\sqrt{2})^p+(a-b\sqrt{2})^p}{2a}$, we have\[
    (a+b\sqrt{2})^p \equiv (-a+b\sqrt{2})^p \pmod{\ell\Z[\sqrt{2}]}.
    \]
    Note that $\ell \neq 2$ since $\frac{x-2}{a}$ is odd, because $x$ is odd by Theorem \ref{thm:elementary 1}. By the same reasoning as before, if $\F = \F_\ell(\theta)$ where $\theta$ is a square root of $2$, then $(a+b\theta)^p = (-a+b\theta)^p$, so if $p \nmid (\abs{\F}-1)$, then $a+b\theta = -a+b\theta$, so $2a \equiv 0 \pmod{\ell}$. Since $\ell \neq 2$, $\ell \mid a$, contradicting that $\gcd\left(a, \frac{x-2}{a}\right) \mid p$.
\end{proof}

\section{On solutions modulo $p$}\label{sec:simple modular observations}

Despite the difficulty in showing that $x^2-2=y^p$ has only trivial solutions, one might hope for partial progress. For instance, one might try to show that all solutions are ``locally trivial.'' The following conjecture is one possible example of this.

\begin{conjecture}\label{conj:triv solns mod p}
    Let $p\geq 3$ be a prime, and let $x,y \in\mathbb{Z}$ such that $x^2-2=y^p$. Then $x\equiv \pm 1 \pmod{p}$ and $y \equiv -1 \pmod{p}$.
\end{conjecture}

Conjecture \ref{conj:triv solns mod p} already seems quite difficult, but it is possible to make some modest progress.

\begin{proposition}
    Let $p\geq 3$ be a prime, and let $x,y \in\mathbb{Z}$ such that $x^2-2=y^p$. Then
    \begin{enumerate}
        \item $x \not \equiv 0 \pmod{p}$,
        \item $y \not \equiv 0 \pmod{p}$.
    \end{enumerate}
\end{proposition}
\begin{proof}
    Let $p$ be an odd prime, and let $x,y \in \mathbb{Z}$ such that $x^2-2=y^p$.

    (1): Assume by way of contradiction that $p \mid x$. By Theorem \ref{thm:only triv solns for p big}, we then have that $p \leq 911$.

    Reducing \eqref{eq:main lebesgue nagell eqn} modulo $p$, we have $y \equiv y^p \equiv -2 \pmod{p}$. Write $y = -2 + ap$ for some $a \in \mathbb{Z}$. We insert this expression for $y$ into \eqref{eq:main lebesgue nagell eqn} and reduce modulo $p^2$. Expanding out with the binomial theorem and simplifying, we find that
    \begin{align}\label{eq:wieferich condition}
        2^{p-1} \equiv 1 \pmod{p^2}.
    \end{align}
    Thus, $p$ is a \emph{Wieferich prime}, which are defined to be those primes $p$ such that \eqref{eq:wieferich condition} holds. There are no Wieferich primes less than 1000 (see \cite{DK2011}, for instance, for a large-scale search for Wieferich primes), so we have a contradiction.

    (2): Assume by way of contradiction that $p \mid y$. Theorem \ref{thm:only triv solns for p big} implies $p \leq 911$. We may also assume $p\geq 17$ by Theorem \ref{thm:thue equations small p}. If $p \mid y$, then $p \| N$ and $p \nmid 128$, so by the modularity theorem we have $a_p(F) \equiv \pm 1 \pmod{p}$. However, we find by computer calculation that $|a_p(F)|^2 \not \equiv 1 \pmod{p}$ for all $5 \leq p < 5000$ (see the file \verb|proposition_8_2.sage| of \cite{Katz_Lebesgue-Nagell_Helper_Code}).
\end{proof}

\section{Explicit formulas for coefficients of newforms} \label{sec:newform coefficients}

The trivial solutions to $x^2-2=y^p$ correspond to certain newforms of level 128. It is possible that studying the coefficients of these newforms could allow for a more effective deployment of modularity techniques. We make here some preliminary comments, without proofs, sketching how one can obtain explicit formulas for some of the coefficients of these newforms.

As noted in Section \ref{sec:simple modular observations}, the four rational newforms of level 128 are all quadratic twists of one another. The twist-minimal form $F$ has label 128.2.a.a on LMFDB \cite{LMFDB-newf-128}. The $q$-expansion of $F$ begins
\begin{align*}
    F(q) = \sum_{n=1}^\infty a_n(F)q^n= q-2q^3-2q^5-4q^7+q^9+2q^{11}-2q^{13}+\cdots,
\end{align*}
and one can show the coefficient of $q^n$ is zero whenever $n$ is even.

The key observation is that $F$ may be written as a linear combination of eta-quotients; this follows from work of Rouse and Webb \cite[Theorem 6]{RW2015}.  (See the introduction of \cite{RW2015} for a definition of \emph{eta-quotient} and related discussion.) Sixteen different eta-quotients appear in the linear combination.

Many of the eta-quotients have coefficients supported on even powers of $q$, which cannot contribute to the coefficients of $F$ by the remark above. 

Several of the eta-quotients have level smaller than 128. For these lower-level forms, one can write the form in terms of Eisenstein series and cusp forms. All the arising cusp forms at lower level have complex multiplication, and their coefficients have relatively simple explicit formulas. For instance, the newform of level 64 is a quadratic twist of the newform of level 32. The explicit formula for the coefficients of the form of level 32 played a key role in a new proof of Watkins on the class number one problem \cite{Wat2019}.

There are two eta-quotients of level 128 that can contribute to the coefficient of $q^n$, $n$ odd. One can obtain formulas for the coefficients of these eta-quotients by using the Jacobi triple product identity (see \cite{And1965}, for instance, for a statement and proof of the triple product identity). For one of these eta-quotients, the coefficients of $q^m$ are nonzero only when $m \equiv 3 \pmod{8}$, starting with $m=11$.

When $p$ is an odd prime, the forms of lower level only contribute to the coefficient of $q^p$ when $p \equiv 3 \pmod{8}$. Hence, when $p \not \equiv 3 \pmod{8}$, the coefficient of $q^p$ in $F$ comes entirely from one eta-quotient of level 128, namely
\begin{align*}
    \frac{\eta(z)^2 \eta(4z)^7}{\eta(2z)^3\eta(8z)\eta(16z)}.
\end{align*}
Using this, one can then show that, if $p$ is an odd prime with $p \not \equiv 3 \pmod{8}$, the coefficient of $q^p$ in $F$ is
\begin{align*}
    a_p(F) = \frac{1}{2}\chi_8(p)\mathop{\sum}_{p=a^2+2b^2+4c^2+8d^2} (-1)^d,
\end{align*}
where $\chi_8$ is the primitive Dirichlet character modulo 8 given by
\begin{align*}
    \chi_8(n) = 
    \begin{cases}
        1, &n \equiv 1,3 \pmod{8}, \\
        -1, &n\equiv 5,7 \pmod{8},
    \end{cases}
\end{align*}
and the sum ranges over all representations of $p$ by the diagonal quaternary quadratic form $a^2+2b^2+4c^2+8d^2$. The coefficient of $q^p$, $p \equiv 3 \pmod{8}$, is much more complicated; it arises from the two eta-quotients of level 128 and several Eisenstein series.

\section{Solutions to the Thue equation modulo $n$}\label{sec:solutions mod n}

In this section, we study the number of local solutions to the Thue equation \eqref{eq4} modulo a positive integer $n$. Because \eqref{eq4} always has the trivial solution $(a, b) = (1, 0)$, we cannot hope to completely rule out solutions this way. However, one could hope that, like what was done in the proof of Theorem \ref{thm: r is 1}, some non-trivial information about the solutions may be gained by combining local information about the solutions to \eqref{eq4} with other techniques.

The first observation to make is that the number of solutions modulo $n$ is a multiplicative function of $n$, with the solutions modulo $mn$ for coprime $m$ and $n$ being the residue classes that are solutions mod $m$ and $n$ separately, so they can be pieced together with the Chinese Remainder Theorem. Thus, we can restrict our attention to $n$ being a prime power.

\begin{theorem}\label{thm:thue equations mod n}
The number of solutions to \eqref{eq4} modulo a prime power $q^s$ is:
\begin{enumerate}

	\item $q^s$, if any of the following conditions hold:
	
	\begin{enumerate}
	
		\item $q \not\in \{p, 2\}$ and $q \not\equiv \pm 1 \pmod{p}$,
		
		\item $q \equiv -1 \pmod{p}$ and $\left(\frac{2}{q}\right) = 1$,
		
		\item $(q, s) = (p, 1)$.
	
	\end{enumerate}
	
	\item Either $q^s+q^{s-1}$ or $q^s+(1-p)q^{s-1}$, if $q \equiv -1 \pmod{p}$ and $\left(\frac{2}{q}\right) = -1$. The first case occurs when $1+\sqrt{2}$ is not a $p$th power in $\mathbb{F}_q(\sqrt{2})$, and the second case occurs when $1+\sqrt{2}$ is a $p$th power in $\mathbb{F}_q(\sqrt{2})$.
	
	\item $2^{s-1}$ if $q = 2$.
	
	\item $pq^{s-1}d$ for some positive integer $d$ depending only on $p, q$, if $q \equiv 1 \pmod{p}$ or $q = p, s > 1$.

\end{enumerate}
\end{theorem}

\begin{proof}

We first find the solutions mod $p$. In this case, \eqref{eq4} reduces to
\begin{align*}
a^p + 2^{\frac{p-1}{2}}b^p \equiv  a + 2^{\frac{p-1}{2}} b \equiv 1 \pmod{p},
\end{align*}
and clearly there is a unique value of $b$ that solves this congruence for every value of $a$. Thus, there are $p$ solutions mod $p$.

\begin{lemma}\label{thm:solutions mod n nondegenerate}
    Assume $q \not\in \{p, 2\}$ and $q \not\equiv 1 \pmod{p}$. Let $n = q^s$. Let $f(x) = f_{1,p}(x)$ be the polynomial corresponding to the Thue equation \eqref{eq4}. Then the number of solutions to \eqref{eq4} mod $n$ is $q^s + (1-r) q^{s-1}$, where $r$ is the number of roots of $f$ mod $q$.
\end{lemma}
\begin{proof}
We examine the values of the polynomial in \eqref{eq4} when $(a, b)$ ranges over elements of the projective line $\proj^1(\Z/n\Z)$. Thus, we say that $(a, b)$ is equivalent to $(ua, ub)$ for $u \in (\Z/n\Z)^\times$. To define the projective line over $\Z/n\Z$, we need $a$ and $b$ to additively generate all of $\Z/n\Z$; this is equivalent to the condition that $q$ does not divide both $a$ and $b$. However, if both $a$ and $b$ are divisible by $q$, then clearly \eqref{eq4} cannot be satisfied. Now, fix an equivalence class, and let $(a, b)$ range over all the representatives of the class. The values of
\begin{align}
\sum_{k=0}^p \binom{p}{k} a^{p-k} b^k 2^{\floor{\frac{k}{2}}} \pmod{n} \label{expr1}
\end{align}
then differ by a $p$th power of an arbitrary unit, since the polynomial in \eqref{expr1} is homogeneous in $a$ and $b$ of degree $p$. Since we are assuming $q \not\equiv 0, 1 \pmod{p}$, the map $u \to u^p$ is bijective on $(\Z/n\Z)^\times$. Thus, there is precisely one solution in the class of $(a, b)$ assuming \eqref{expr1} is invertible.

We now investigate when the sum is not invertible, or equivalently, when \eqref{expr1} is $0$ mod $q$. If $q \mid a$, then \eqref{expr1} is $\equiv b^p 2^{\frac{p-1}{2}}$ mod $q$, which is not $0$ modulo $q$ since in this case $b$ must be invertible and $q \neq 2$. Similarly, if $q \mid b$, then \eqref{expr1} is $\equiv a^p \pmod{q}$, which is not $0$ modulo $q$. Therefore, the only classes for which \eqref{expr1} could not be invertible are those for which both $a$ and $b$ are invertible. In this case, we set $c = ab^{-1}$, which is well-defined on $\proj^1(\Z/n\Z)$. Factoring out $b^p$ in \eqref{expr1}, we must have that $c$ is a root of $f(x)$ modulo $q$. For any class for which $c$ is a root of $f(x)$, there are no solutions, but otherwise, there is a unique solution. The number of $(a, b)$ equivalence classes is $q^s + q^{s-1}$, and if there are $r$ roots of $f$ modulo $q$, then there are $rq^{s-1}$ classes for which $c$ is a root, so there are $q^s + (1-r)q^{s-1}$ solutions total.
\end{proof}

Assume the hypotheses of Lemma \ref{thm:solutions mod n nondegenerate}. We now show that under the additional assumption that $q \not\equiv -1 \pmod{p}$ or $\left(\frac{2}{q}\right) = 1$, then $r = 1$. In this case, if $\F$ is the field $\F_q(\sqrt{2})$, then $p \nmid (\abs{\F}-1)$, so the map $x \mapsto x^p$ is bijective on $\F$. Therefore, we can uniquely take $p$th roots in $\F$. Also, fixing a choice of $\sqrt{2} \in \F$, we recall $f$ can be written as
\begin{align*}
f(x) &= \frac{1}{2\sqrt{2}}\left((1+\sqrt{2})(x+\sqrt{2})^p-(1-\sqrt{2})(x-\sqrt{2})^p\right),
\end{align*}
and using the same manipulations as those leading to \eqref{eq501}, we find that there is a unique root
\begin{align*}
\theta &= -\sqrt{2} \frac{(\sqrt{2}+1)^{1/p} - (\sqrt{2}-1)^{1/p}}{(\sqrt{2}+1)^{1/p}+(\sqrt{2}-1)^{1/p}}.
\end{align*}
Now, if $\left(\frac{2}{q}\right) = 1$, then we are done since we have shown $f$ has a unique root in $\F_q$. If $\left(\frac{2}{q}\right) = -1$, then we must show that $\theta$ lies in $\F_q$. But since the coefficients of $f$ are in $\F_q$, any Galois conjugate of $\theta$ must also be a root of $f$. Since there is only one root, $\theta$ is fixed under $\gal(\F/\F_q)$, hence must lie in $\F_q$. Therefore, there is exactly one root in $\F_q$, so $r = 1$. Thus, we have proven (1).

Now, if $q \equiv -1 \pmod{p}$ and $\left(\frac{2}{q}\right) = -1$, the hypotheses of Lemma \ref{thm:solutions mod n nondegenerate} are still satisfied, but we are guaranteed neither existence nor uniqueness of $p$th roots in $\F = \F_q(\sqrt{2})$. In fact, since $\F^\times \cong \Z/(q^2-1)\Z$, if a $p$th root of an element of $F^\times$ exists, then the number of $p$th roots is $p$. Therefore, the manipulations of \eqref{eq501} either give $0$ or $p$ roots of the form
\begin{align*}
\sqrt{2} - \frac{2\sqrt{2}}{\sqrt[p]{(1+\sqrt{2})^2}+1}.
\end{align*}
Since $2$ is coprime to $p$, $(1+\sqrt{2})^2$ is a $p$th power in $\F$ if and only if $1+\sqrt{2}$ is a $p$th power in $\F$ (since if $(1+\sqrt{2})^2 = c^p$, then $\left((1+\sqrt{2})c^{-\frac{p-1}{2}}\right)^p = 1+\sqrt{2}$). If $1+\sqrt{2}$ has a $p$th root $x+y\sqrt{2}$, then since $(1+\sqrt{2})(1-\sqrt{2}) = -1$, $(x+y\sqrt{2})(x-y\sqrt{2})$ is $-1$ times a $p$th root of unity. But this $p$th root of unity is fixed under the nontrivial automorphism $\sigma$ that sends $\sqrt{2}$ to $-\sqrt{2}$, so it must be in $\F_q$, and since $p \nmid (q-1)$ it must be $1$. Therefore, $(x+y\sqrt{2})^{-1} = -x+y\sqrt{2}$. Then if $\zeta$ is a primitive $p$th root of unity in $\F$, the roots of $f$ are of the form
\begin{align*}
\rho_i &= \sqrt{2} - \frac{2\sqrt{2}}{(x+y\sqrt{2})^2\zeta^i + 1}= \sqrt{2} \frac{(x+y\sqrt{2})^2 \zeta^i - 1}{(x+y\sqrt{2})^2 \zeta^i + 1}= \sqrt{2} \frac{(x+y\sqrt{2})\zeta^i + x-y \sqrt{2}}{(x+y\sqrt{2})\zeta^i - x+y \sqrt{2}}.
\end{align*}
Now, $\zeta  \sigma(\zeta)$ is a $p$th root of unity in $\F_q$, so $\sigma(\zeta) = \zeta^{-1}$. Therefore,
\begin{align*}
\sigma(\rho_i) &= -\sqrt{2} \frac{(x-y\sqrt{2})\zeta^{-i}+x+y\sqrt{2}}{(x-y\sqrt{2})\zeta^{-i} - x-y\sqrt{2}}= \sqrt{2} \frac{x+y\sqrt{2}+(x-y\sqrt{2})\zeta^{-i}}{x+y\sqrt{2}-(x-y\sqrt{2})\zeta^{-i}}= \rho_i.
\end{align*}
Therefore, $\rho_i \in \F_q$. Therefore, either $1+\sqrt{2}$ is not a $p$th power in $\F$ and $r = 0$, or $1+\sqrt{2}$ is a $p$th power in $\F$ and $r = p$. Thus, we have proven (2).

Now suppose $q = 2$. Thus, we are looking at solutions to
\begin{align*}
\sum_{k=0}^p \binom{p}{k} a^{p-k} b^k 2^{\floor{\frac{k}{2}}} = 1 \pmod{2^s}.
\end{align*}
Again, not both $a$ and $b$ are even, so we look at equivalence classes in $\proj^1(\Z/2^s\Z)$. Each class has a unique solution if the sum is invertible and no solutions otherwise. If $b$ is odd, then factoring out a power of $b^p$ and letting $c = a b^{-1}$, the sum is equal to
\begin{align*}
\sum_{k=0}^p \binom{p}{k} c^{p-k} 2^{\floor{\frac{k}{2}}} &\equiv c^p + p c^{p-1} = c^{p-1}(c+p) \pmod{2},
\end{align*}
up to a $p$th power of a unit. Since $p$ is odd, this is $0$ regardless of what $c$ is. Therefore, there are no solutions for $b$ odd. Otherwise, we must have $a$ odd, $b$ even. Then factoring out a power of $a^p$ and letting $c = ba^{-1} \equiv 0 \pmod{2}$, the sum is equal to
\begin{align*}
\sum_{k=0}^p \binom{p}{k} c^k 2^{\floor{\frac{k}{2}}} \equiv 1 \pmod{2},
\end{align*}
up to a $p$th power of a unit. Hence, the sum is invertible. Therefore, the equivalence classes with solutions are precisely those of the form $(1, 2k)$, each with a unique solution. Hence, there are $2^{s-1}$ solutions mod $2^s$, and we have proven (3).

Lastly, suppose $q = p$ and $s \geq 2$, or $q \equiv 1 \pmod{p}$. In this case, for each class in $\proj^1(\Z/n\Z)$ represented by $(a, b)$, since the values of the sum
\begin{align*}
\sum_{k=0}^p \binom{p}{k} a^{p-k} b^k 2^{\floor{\frac{k}{2}}}
\end{align*}
differ by $p$th powers of units, there will be $p$ solutions if the sum is a unit $p$th power mod $q^s$, and $0$ solutions otherwise. Hence, the number of solutions is $pr$ where $r$ is the number of equivalence classes in $\proj^1(\Z/n\Z)$ such that the sum is a unit $p$th power. Equivalently, $r$ is the number of values of $c$ mod $q^s$ such that $\sum_{k=0}^p \binom{p}{k} c^{p-k} 2^{\floor{\frac{k}{2}}}$ is a unit $p$th power plus the number of values of $c$ mod $q^{s-1}$ such that $\sum_{k=0}^p \binom{p}{k} (cq)^k 2^{\floor{\frac{k}{2}}}$ is a unit $p$th power mod $q^s$. It remains to show that $r = q^{s-1} d$ for a positive integer $d$ depending only on $p$ and $q$.

In the case that $q = p, s \geq 2$, note that $x$ is a unit $p$th power mod $p^s$ if and only if $x$ is a unit $p$th power mod $p^2$. This can be proven by a straightforward counting argument, since there are $p^{s-2}(p-1)$ unit $p$th powers mod $p^s$, each maps to a unit $p$th power mod $p^2$, there are $p-1$ options to which to map, and at most $p^{s-2}$ can map to each one. Similarly, when $q \equiv 1 \pmod{p}$, $x$ is a unit $p$th power mod $q^s$ if and only if $x$ is a unit $p$th power mod $q$. This also follows from a counting argument, since there are $q^{s-1}(q-1)/p$ unit $p$th powers mod $q^s$, each one maps to a unit $p$th power mod $q$, there are $(q-1)/p$ options to map to, and at most $q^{s-1}$ can map to each one. (One could also use Hensel lifting to obtain these results.)

For the case $q = p, s \geq 2$, note that $\sum_{k=0}^p \binom{p}{k} (cp)^k 2^{\floor{\frac{k}{2}}} \equiv 1 \pmod{p^2}$ is always a unit $p$th power. This contributes $p^{s-1}$ to $r$. Next, we note that the value of $\sum_{k=0}^p \binom{p}{k} c^{p-k} 2^{\floor{\frac{k}{2}}}$ mod $p^2$ only depends on the value of $c\equiv c_0$ mod $p$. For $1\leq k \leq p-1$, we have $p \mid {p \choose k}$ and $c^{p-k} \equiv c_0^{p-k} \pmod{p}$. The $k=p$ term does not depend on $c$ at all, and for the $k=0$ term, $(c_0+jp)^p = \sum_{k=0}^p \binom{p}{k} c_0^k (jp)^{p-k} \equiv c_0^p \pmod{p^2}$. Each value of $c_0$ mod $p$ for which $\sum_{k=0}^p \binom{p}{k} c^{p-k} 2^{\floor{\frac{k}{2}}}$ is a unit $p$th power mod $p^2$ has $p^{s-1}$ lifts mod $p^s$. Therefore, if we let $d$ be $1$ plus the number of such values of $c_0$, we have $r = p^{s-1} d$.

For the case $q \equiv 1 \pmod{p}$, note that $\sum_{k=0}^p \binom{p}{k} (cq)^k 2^{\floor{\frac{k}{2}}} \equiv 1 \pmod{q}$ is always a unit $p$th power. This contributes $q^{s-1}$ to $r$. Each value of $c$ mod $q$ for which $\sum_{k=0}^p \binom{p}{k} c^{p-k} 2^{\floor{\frac{k}{2}}}$ is a unit $p$th power has $q^{s-1}$ lifts mod $q^s$. Therefore, if we let $d$ be $1$ plus the number of such values of $c$, we have $r = q^{s-1} d$.
\end{proof}

\section*{Acknowledgments}

We thank Nick Andersen for valuable conversations and computations regarding modular forms and eta-quotients. We thank Maurice Mignotte and Samir Siksek for helpful correspondence regarding their work on the Lebesgue--Nagell equation studied in this paper. We are grateful to Michael Bennett for bringing the reference \cite{BPV2024} to our attention, which allowed us to strengthen some of the results in this paper. We thank the anonymous referees for pointing out the reference \cite{Barros2010}, and for helpful suggestions.

The second author is partially supported by the National Science Foundation (DMS-2418328) and the Simons Foundation (MPS-TSM-00007959).

\bibliographystyle{plain}
\bibliography{refs}

\end{document}